\documentclass[11pt, reqno]{amsart}
\usepackage{mathrsfs}
\usepackage[centertags]{amsmath}
\usepackage{amssymb,array}
\usepackage{amsthm}
\usepackage{stmaryrd}
\DeclareFontFamily{U}{stmry}{}
\DeclareFontShape{U}{stmry}{m}{n}{<-> stmary10}{}
\DeclareFontShape{U}{stmry}{b}{n}{<->ssub * stmry/m/n}{}
\usepackage{graphicx}
\usepackage{ytableau}
\usepackage{enumerate,enumitem}
\usepackage{MnSymbol}
\SetLabelAlign{center}{\hfill #1\hfill}
\usepackage[textwidth=16cm, textheight=24cm, hmarginratio=1:1,vmarginratio=1:1]{geometry}
\usepackage{tikz-cd, tikz}
\usepackage{rotating}
\usepackage[all,cmtip]{xy}
\usepackage[titletoc, title]{appendix}
\usepackage{amscd}
\usepackage[colorlinks]{hyperref} 
\usepackage{float}
\usepackage{mathtools}
\usepackage{dsfont} 
\hypersetup{
bookmarksnumbered,
pdfstartview={FitH},
breaklinks=true,
linkcolor=blue,
urlcolor=blue,
citecolor=blue,
bookmarksdepth=2
}
\usepackage{adjustbox}
\usepackage[normalem]{ulem}

\theoremstyle{plain}
   
   \newtheorem{theorem}{Theorem}[section]
   \newtheorem{proposition}[theorem]{Proposition}
   
   \newtheorem{lemma}[theorem]{Lemma}
   
   \newtheorem{conjecture}[theorem]{Conjecture}
   
\theoremstyle{definition}
   \newtheorem{definition}[theorem]{Definition}
   
   \newtheorem{example}[theorem]{Example}

   \newtheorem{remark}[theorem]{Remark}

\numberwithin{equation}{section}

\usetikzlibrary{arrows,automata,positioning}
\newcommand{\midarrow}{\tikz \draw[thin,-angle 45] (0,0) -- +(.25,0);}

\newcommand{\CC}{{\mathbb {C}}}

\newcommand{\ZZ}{{\mathbb {Z}}}

\newcommand{\SSYT}{{\rm SSYT}}

\newcommand{\res}{{\rm res}}
\newcommand{\Rep}{{\rm Rep}}

\DeclareMathOperator*{\diag}{diag}
\DeclareMathOperator{\Gr}{Gr}

\begin{document}

\title[Monoidal categorification of generalized cluster algebras]{Monoidal categorification of generalized cluster algebras and conjectures of Fraser and Gleitz}

\author{Xiao-Juan An, Jian-Rong Li, and Yan-Feng Luo$^\dag$}\thanks{$\dag$ Corresponding author}

\address{Xiao-Juan An: School of Mathematics and Statistics, Lanzhou University, Lanzhou 730000, P. R. China; School of Information Engineering and Artificial Intelligence, Lanzhou University of Finance and Economics, Lanzhou 730010, P. R. China.}
\email{anxj0820@163.com}

\address{Jian-Rong Li: Faculty of Mathematics, University of Vienna, Oskar-Morgenstern Platz 1, 1090 Vienna, Austria.}
\email{lijr07@gmail.com}

\address{Yan-Feng Luo: School of Mathematics and Statistics, Lanzhou University, Lanzhou 730000, P. R. China.}
\email{luoyf@lzu.edu.cn}
\date{}

\begin{abstract}
Hernandez and Leclerc introduced the notion of monoidal categorification of cluster algebras. We define similarly the notion of monoidal categorifications of generalized cluster algebras: an abelian monoidal category $\mathcal M$ is said to be a monoidal categorification of a generalized cluster algebra $\mathcal A$ if the Grothendieck ring of $\mathcal M$ is isomorphic to the upper generalized cluster algebra $\mathcal A^{\mathrm{up}}$, and if cluster monomials (resp. cluster variables) of $\mathcal A$ correspond to classes of real simple (resp. real prime simple) objects of $\mathcal M$.

Let $\varepsilon$ be a root of unity such that $\varepsilon^{2\ell}=1$ for some $\ell\in\mathbb{Z}_{\geq 2}$. Denote by $\mathcal{C}_{\varepsilon}$ the category of finite-dimensional modules of the restricted quantum loop algebra $U_\varepsilon^{\res}(L\mathfrak{sl}_k)$ at root $\varepsilon$ of unity, and let $\mathcal{C}_{\varepsilon, \xi}$ be a full subcategory of $\mathcal{C}_{\varepsilon}$ determined by a bipartition $\xi: I \to \{0,1\}$ of the Dynkin diagram. For $k=3$, Gleitz conjectured that the Grothendieck ring of $\mathcal C_{\varepsilon,\xi}$ is isomorphic to a generalized cluster algebra of rank $2\ell-2$, and that generalized cluster monomials correspond to classes of simple modules. This conjecture is a special case of a more general conjecture of Fraser.
 
In this paper, we prove the first part of Gleitz's conjecture. More precisely, for $k=3$ and arbitrary $\ell\ge2$, we prove that the Grothendieck ring of $\mathcal C_{\varepsilon,\xi}$ is isomorphic to a generalized cluster algebra of rank $2\ell-2$. We also classify the real Kirillov--Reshetikhin modules of $U^{\mathrm{res}}_\varepsilon(L\mathfrak{sl}_3)$ and obtain mutation sequences for the real Kirillov--Reshetikhin modules from the initial seed of the generalized cluster algebra.

\hspace{0.15cm}


\hspace{0.15cm}

\end{abstract}

\maketitle 
\setcounter{tocdepth}{1}
\tableofcontents

\section{Introduction}
\subsection{Restricted quantum loop algebras at roots of unity}
Let $\mathfrak{g}$ be a simple Lie algebra over $\mathbb{C}$ and let $\mathfrak{\widehat{g}}$ be the associated affine Lie algebra and $U_q(\widehat{\mathfrak{g}})$ the corresponding quantum affine algebra. Denote by $U_q(L \mathfrak{g})$ the quantum loop algebra, which is isomorphic to a quotient of $U_q(\widehat{\mathfrak{g}})$ where the central charge is mapped to $1$. Denote by $U_q^{\res}({L\mathfrak{g}})$ the $\mathbb{C}[q,q^{-1}]$-subalgebra of $U_q(L\mathfrak{g})$ generated by the $q$-divided powers of the Chevalley generators {\cite[Section 1]{CP97}}. Let $\varepsilon$ be a root of unity. Let $U_\varepsilon^{\res}({L\mathfrak{g}})$ be the specialization of $U_q^{\res}({L\mathfrak{g}})$ at $\varepsilon$, by setting
\begin{align*}
U_\varepsilon^{\res}({L\mathfrak{g}}):=U_q^{\res}({L\mathfrak{g}}) \otimes_{\mathbb{C}[q, q^{-1}]}\mathbb{C}
\end{align*}
via the algebra homomorphism $f_{\varepsilon}: \mathbb{C}[q, q^{-1}]\rightarrow \mathbb{C}$, that takes $q$ to $\varepsilon$.

In 1997, Chari and Pressley \cite{CP97} classified the finite-dimensional irreducible modules of $U_{\varepsilon}^{\res}(L\mathfrak{g})$ in terms of highest weights, where $\varepsilon$ is a root of unity of odd order. In 2002, Frenkel and Mukhin \cite{FM02} extended the Chari and Pressley's result to all roots of unity. Denote by $I$ the set of vertices of the Dynkin diagram of $\mathfrak{g}$. The finite-dimensional simple modules of $U_{\varepsilon}^{\res}(L\mathfrak{g})$ are classified by $(P_i(u))_{i \in I}$, see \cite[Theorem 8.2]{CP97} and \cite[Theorem 2.4]{FM02}, where every $P_i(u) \in \mathbb{C}[u]$ is a polynomial with constant term $1$ which is called a Drinfeld polynomial. Every $I$-tuple $(P_i(u))_{i \in I}$ of Drinfeld polynomials corresponds to a dominant monomial $m$ in formal variables $Y_{i,a}$, $i \in I$, $a \in \mathbb{C}^{\times}$, where dominant means the exponents of $Y_{i,a}$ in $m$ are non-negative. Denote by $L(m)$ the corresponding simple $U_{\varepsilon}^{\res}(L\mathfrak{g})$-module.

\subsection{Generalized cluster algebras and their monoidal categorifications}

Fomin and Zelevinsky introduced the notion of cluster algebras \cite{FZ02} as an algebraic framework for studying total positivity in algebraic groups and canonical bases in quantum groups.

As a natural generalization, Chekhov and Shapiro introduced the generalized cluster algebras which arise from the Teichm\"{u}ller spaces of Riemann surfaces with orbifold points \cite{CS14}. The principal distinction between cluster algebras and generalized cluster algebras lies in the fact that the binomial exchange relations for cluster variables of cluster algebras are replaced by the multinomial exchange relations for cluster variables of generalized cluster algebras. In \cite{CS14}, Chekhov and Shapiro have shown that the generalized cluster algebras possess the remarkable Laurent phenomenon. Many other important results for cluster algebras extend to the generalized setting, such as the finite-type classification, the existence and theory of g-vectors, and the construction of $F$-polynomials \cite{CS14, N15}.

A simple module $M$ of $U_q(L\mathfrak{g})$ is called real if the tensor product $M \otimes M$ is simple, see \cite{L03}. Similarly, we say that a simple module $M$ of $U_\varepsilon^{\rm res}(L\mathfrak{g})$ is real if $M \otimes M$ is simple.

A simple module $M$ of $U_q(L\mathfrak{g})$ is called prime if there are no non-trivial 
modules $M_1, M_2$ of $U_q(L\mathfrak{g})$ such that $M \cong M_1 \otimes M_2$, see \cite{CP97b}. Similarly, we say that a simple module $M$ of $U_\varepsilon^{\rm res}(L\mathfrak{g})$ is prime if there are no non-trivial modules $M_1, M_2$ of $U_\varepsilon^{\rm res}(L\mathfrak{g})$ such that $M \cong M_1 \otimes M_2$.

Hernandez and Leclerc \cite[Definition 2.1]{HL10} introduced the concept of monoidal categorifications of cluster algebras. Let $\mathcal{A}$ be a cluster algebra and let $\mathcal{M}$ be an abelian monoidal category. The category $\mathcal{M}$ is called a monoidal categorification of $\mathcal{A}$ if the Grothendieck ring of $\mathcal{M}$ is isomorphic to $\mathcal{A}$, and if 
\begin{enumerate}
\item the cluster monomials of $\mathcal{A}$ are the classes of all the real simple objects of $\mathcal{M}$;
\item the cluster variables of $\mathcal{A}$ (including the frozen ones) are the classes of all the real prime simple objects of $\mathcal{M}$. 
\end{enumerate} 

We propose the following definition of monoidal categorification of generalized cluster algebras. 
\begin{definition}
Let $\mathcal{A}$ be a generalized cluster algebra and let $\mathcal{M}$ be an abelian monoidal category. We say that the category $\mathcal{M}$ is a monoidal categorification of $\mathcal{A}$ if the Grothendieck ring of $\mathcal{M}$ is isomorphic to the upper generalized cluster algebra $\mathcal{A}^{\rm up}$ of $\mathcal{A}$, and if 
\begin{enumerate}
\item the cluster monomials of $\mathcal{A}$ correspond to classes of real simple objects of $\mathcal{M}$;
\item the cluster variables of $\mathcal{A}$ correspond to classes of real prime simple objects of $\mathcal{M}$. 
\end{enumerate} 
\end{definition}

Unlike the case of monoidal categorification of cluster algebras, we do not require that frozen variables of $\mathcal{A}$ are real objects of $\mathcal{M}$. Moreover, we require only that cluster monomials (resp. cluster variables) correspond to real (resp. real prime) objects, but not conversely that all real (resp. real prime) simple modules arise as cluster monomials (resp. cluster variables). We further require only that $\mathcal{A}^{\rm up}$, rather than $\mathcal{A}$ itself, be isomorphic to the Grothendieck ring of $\mathcal{M}$ (see Section~\ref{sec:example in type A3} for an example). 

\subsection{Gleitz’s conjecture and Fraser's conjecture} \label{subsec:introduction:a proof of Gleitz conjecture}

In this paper, unless we say otherwise, for any integer $\ell\geq 2$, denote by $\varepsilon$ the root of unity
\begin{align*}
\varepsilon:=
\begin{cases} \text{exp}(\frac{i\pi}{\ell}), & \text{if}{\hskip 0.4em} \ell {\hskip 0.4em}\text{is even},\\
              \text{exp}(\frac{2i\pi}{\ell}), & \text{if}{\hskip 0.4em} \ell {\hskip 0.4em} \text{is odd}.
\end{cases}
\end{align*}
The integer $\ell$ is the order of $\varepsilon^{2}$, and we have $\varepsilon^{2 \ell}=1$.

Denote by $\mathcal{C}_{\varepsilon}$ the category of finite-dimensional modules of the restricted quantum loop algebra $U_\varepsilon^{\res}({L\mathfrak{sl}_k})$ at root $\varepsilon$ of unity. 

Denote by $\mathcal{P}$ the free abelian group in formal variables $Y^{\pm 1}_{i,s}$, $i\in I, s\in \ZZ$ and $\mathcal{P}^+$ is the submonoid of $\mathcal{P}$ generated by $Y_{i,s}$, $i\in I, s\in \ZZ$. Let $\varepsilon^{2 \ell}=1$ with $\ell\in \ZZ_{\geq 2}$. We define $\mathcal{P}^+_{\varepsilon}$ to be the quotient monoid
\[
\mathcal{P}^+_{\varepsilon}
:=
\mathcal{P}^+ \big/ \sim,
\]
where $\sim$ is the smallest monoid congruence generated by the relations
\[
Y_{i,p} = Y_{i,p+2\ell},
\qquad
\text{for all } i\in I,\; p\in \mathbb{Z}.
\]
Let $I = I_0 \sqcup I_1$ be a bipartition of $I$ and define $\xi_i=0$ for $i \in I_0$, and $\xi_i=1$ for $i \in I_1$. 
We define $\mathcal{P}^+_{\varepsilon,\xi}$ to be the submonoid of
$\mathcal{P}^+_{\varepsilon}$ generated by $Y_{i,s}$, $i\in I$, $s \equiv \xi_i \pmod{2}$.
When writing elements of $\mathcal{P}^+_{\varepsilon}$ and $\mathcal{P}^+_{\varepsilon, \xi}$, we always choose their representatives whose second indices lie in the interval
$[0,\,2\ell-1]$.

Denote by $\mathcal{C}_{\varepsilon, \xi}$ the full subcategory of $\mathcal{C}_{\varepsilon}$ whose objects $M$ have all their composition factors $L(m)$, $m \in \mathcal{P}^+_{\varepsilon, \xi}$, see \cite[Section 3.3]{Gle16}. Without loss of generality, for $\xi: I = I_0 \cup I_1 \to \{0,1\}$, we choose $I_0 =\{ i\in I \mid i \text{ is odd}\}$, $I_1 =\{ i\in I \mid i \text{ is even}\}$. In the case of $k=2$, $\ell \in \ZZ_{\ge 2}$ and in the case of $k=3$, $\ell=2$, Gleitz showed that the Grothendieck ring $K_0(\mathcal{C}_{\varepsilon, \xi})$ is a generalized cluster algebra structure of finite Dynkin types $C_{\ell-1}$ and $G_2$, respectively \cite[Theorem 4.1 and 5.3]{Gle16}. Moreover, Gleitz gave the following conjecture.

\begin{conjecture} [{\cite[Conjecture 5.4]{Gle16}}] \label{conj: Gleitz conjecture}
For $k=3$ and $\ell \in \mathbb{Z}_{\geq 2}$, the Grothendieck ring of $\mathcal{C}_{\varepsilon, \xi}$ is isomorphic to a generalized cluster algebra of rank $2 \ell-2$. Moreover, the generalized cluster monomials are mapped to classes of simple modules.
\end{conjecture}

Fraser made the following conjecture, which agrees with Gleitz's results and conjectures and extended them to arbitrary $k$ and $\ell$.
\begin{conjecture}[{\cite[Conjecture 8.8]{Fra20}}]
The Grothendieck group $K_0(\mathcal{C}_{\varepsilon, \xi})$ admits an upper generalized cluster algebra structure in which each cluster monomial is the class of a simple module. An initial seed for this cluster algebra can be obtained by setting all frozen variables $x_{N+i} = 1$ and identifying the mutable variables $x_i$ and coefficient string variables $z_s$ with the classes of appropriate simple modules, where $x_j$ and $z_s$ are defined in Definition~\ref{defn:abstract CS seed}. 
\end{conjecture}

The aim of this paper is to study monoidal categorification of generalized cluster algebras at roots of unity and to prove the first part of Gleitz's conjecture. More precisely, for $k=3$, $\ell\in \mathbb Z_{\ge2}$, and $\varepsilon^{2\ell}=1$, we prove that the Grothendieck ring of the category $\mathcal C_{\varepsilon,\xi}$ is isomorphic to a generalized cluster algebra of rank $2\ell-2$ (see Theorem~\ref{Th:isomorphic}).

To establish this result, we first develop a combinatorial parametrization of dominant monomials in $\mathcal C_{\varepsilon,\xi}$ using semistandard Young tableaux. More precisely, we prove that the monoid $\mathcal{P}^+_{\varepsilon, \xi}$ of dominant monomials is naturally isomorphic to the monoid $\SSYT(k, [k+\ell], \sim)$ of equivalences of semistandard Young tableaux (see Theorem~\ref{thm: isomorphic tableaux}). This allows us to reformulate Fraser's conjecture in the language of quantum affine algebras.

In contrast with the generic quantum parameter case, Kirillov--Reshetikhin modules of the restricted quantum loop algebra $U_\varepsilon^{\res}(L\mathfrak g)$ at a root of unity are not necessarily real. We classify the real Kirillov--Reshetikhin modules of $U_\varepsilon^{\res}(L\mathfrak{sl}_3)$ and show that a Kirillov--Reshetikhin module is real precisely when its level is strictly smaller than $\ell$ (see Theorem~\ref{Th: real KR module}).

We then construct explicit mutation sequences for the real Kirillov--Reshetikhin modules from the initial seed of the generalized cluster algebra (see Theorem~\ref{Th:KR mutation sequence}). Using these mutation sequences together with a description of the image of the $\varepsilon$-character homomorphism as an intersection of kernels of screening operators, we prove that the Grothendieck ring of $\mathcal C_{\varepsilon,\xi}$ is isomorphic to the corresponding generalized cluster algebra, thereby establishing the first part of Gleitz's conjecture, see Section \ref{subsec: proof of isomorphism theorem}.

\subsection{Organization of the paper}

In Section \ref{Section: Generalized cluster algebras and representations of restricted quantum affine algebras at roots of unity}, we recall background on generalized cluster algebras, restricted quantum loop algebras $U_\varepsilon^{\res}({L\mathfrak{g}})$ at roots of unity, finite-dimensional $U_\varepsilon^{\res}({L\mathfrak{g}})$-modules and their $\varepsilon$-characters. In Section \ref{Section: Fraser's conjecture and Gleitz's conjecture}, we prove that the monoid $\mathcal{P}^+_{\varepsilon, \xi}$ is isomorphic to the monoid ${\rm SSYT}(k, [k+\ell], \sim)$ of equivalence classes of semistandard tableaux, see Theorem \ref{thm: isomorphic tableaux}. We recall cyclic symmetry loci in Grassmannians, Fraser's conjecture \cite{Fra20}, and Gleitz's conjecture \cite{Gle16}. We rewrite Fraser's conjecture in the language of quantum affine algebras. In Section \ref{Section: main result}, we present the main results of this paper, including a classification of real Kirillov--Reshetikhin modules of $U_\varepsilon^{\res}(L\mathfrak{sl}_3)$, mutation sequences of real Kirillov--Reshetikhin modules of $U_\varepsilon^\res({L\mathfrak{sl}_{3}})$, and an isomorphism between the Grothendieck ring $\mathcal{C}_{\varepsilon, \xi}$ and a generalized cluster algebra of rank $2 \ell-2$. In Section \ref{sec:example in type A3}, we explain that, in defining a monoidal categorification of a generalized cluster algebra $\mathcal{A}$, one should require that the upper cluster algebra $\mathcal{A}^{\mathrm{up}}$, rather than $\mathcal{A}$ itself, be isomorphic to the Grothendieck ring of $\mathcal{M}$.

\section{Generalized cluster algebras and restricted quantum loop algebras at roots of unity}\label{Section: Generalized cluster algebras and representations of restricted quantum affine algebras at roots of unity}
 In this section, we recall results about generalized cluster algebras \cite{BCDX20, CS14}, restricted quantum loop algebras $U_\varepsilon^{\res}({L\mathfrak{g}})$ at roots of unity \cite{CP97, FM02}. 
 \subsection{Generalized cluster algebras}
 We recall terminology, definitions and related results of the generalized cluster algebras of geometric types \cite{BCDX20}.
 For integers $i \le j$, we denote $[i,j] = \{i,i+1,\ldots,j\}$ and for $i \in \ZZ_{\ge 1}$, we denote $[i] = [1,i]$. An $n\times n$ matrix $B$ is called skew-symmetrizable if there exists a diagonal matrix $\widetilde{D}=\diag (\widetilde{d}_1, \widetilde{d}_2, \ldots, \widetilde{d}_n )$, where $\widetilde{d}_i$ are positive integers for all $i\in [1,n]$, such that $\widetilde{D}B$ is skew-symmetric.

Let $(\mathbb{P},\cdot,\oplus)$ be a commutative semifield, referred to as the coefficient group. For example, one may take $\mathbb{P}$ to be the tropical semifield $\mathrm{Trop}(\lambda_1,\dots,\lambda_n)$ generated by the indeterminates $\lambda_1,\dots,\lambda_n$. By definition, $\mathrm{Trop}(\lambda_1,\dots,\lambda_n)$ is the set of Laurent monomials in the $\lambda_i$'s, equipped with the usual multiplication and the tropical addition
 \begin{equation*}
\left(\displaystyle\prod_i \lambda_i^{a_i}\right) \oplus \left(\displaystyle\prod_i \lambda_i^{b_i}\right) = \left(\displaystyle\prod_i \lambda_i^{\min(a_i,b_i)}\right).
 \end{equation*}
Let $\mathcal{F}=\mathbb{ZP}(t_1,\dots,t_n)$ be the ambient field of rational functions in $n$ independent variables, where $\mathbb{ZP}$ is the integer group ring of $\mathbb{P}$.

For convenience, denote $[x]_+:=x$ if $x \geq 0$, and $[x]_+:=0$ otherwise.

For every $i \in [1,n]$, $d_i \in \ZZ_{\ge 1}$, and a collection of variables $\rho_i=(\rho_{i,0}, \rho_{i,1},\dots, \rho_{i,d_i})\in\mathbb{P}^{d_i+1}$, where $\rho_{i,0}=\rho_{i,d_i}=1$, define the corresponding homogeneous \emph{exchange polynomial}
\begin{equation} \label{eq:theta_i exchange polynomial}
\theta_i[\rho_i](u,v) := \displaystyle \sum_{r=0}^{d_i} \rho_{i,r} u^r v^{d_i-r}\in\mathbb{ZP} [u,v].
\end{equation}
The integer $d_i$ is called the exchange degree of the polynomial $\theta_i[\rho_i](u,v)$.  

\begin{definition}[{\cite[Definition 2.1]{CS14}}] \label{def of generalized seed}
With above notations, a generalized seed of geometric type of rank $n$ is defined to be a triple $\big(\widetilde{\mathbf{x}},\rho,\widetilde{B}\big)$, where  
\begin{itemize}
\item[(1)] $\widetilde{\mathbf{x}}=\{x_{1},x_2,\ldots, x_{m}\}$ is called an extended cluster, $\mathbf{x}=\{x_{1},x_2,\ldots,x_{n}\}$ is called a cluster, $x_1, x_2, \ldots,x_{n}$ are called cluster variables and $x_{n+1},\ldots,x_{m}$ are called frozen variables;
  
\item[(2)] $\rho=\{\rho_{1},\rho_2,\ldots,\rho_{n}\}$, where for each $i\in[1,n]$, $\rho_i=(\rho_{i,0},\rho_{i,1},\dots,\rho_{i,d_i}) \in\mathbb{P}^{d_i+1}$ consists of the coefficients of the $i$-th exchange polynomial $\theta_i$ in (\ref{eq:theta_i exchange polynomial});
 
\item[(3)] $\widetilde{B}=(b_{ij})$ be an $m\times n$ integer matrix whose top $n\times n$ submatrix $B$ is skew-symmetrizable. 
\end{itemize}
\end{definition}
The matrix $\widetilde{B}$ is called an extended exchange matrix and the top $n\times n$ submatrix $B$ of $\widetilde{B}$ is called the principal part of $\widetilde{B}$.

\begin{definition}[{\cite[Definition 2.2]{CS14}}, {\cite[Definition 2.6]{Fra20}}]\label{Def: the mutation of a generalized cluster seed}
For $i\in [1,n]$, the mutation of a generalized seed $\big(\widetilde{\mathbf{x}},\rho,\widetilde{B}\big)$ in the direction $i$ is another generalized seed \smash{$\mu_i\big(\widetilde{\mathbf{x}},\rho,\widetilde{B}\big):=\big(\widetilde{\mathbf{x}}',
\rho',\widetilde{B}'\big)$}, where
\begin{itemize}
\item[(1)] $\widetilde{\mathbf{x}}':=(\widetilde{\mathbf{x}}- \{x_i\})\cup\{x'_{i}\}$ is given by
\begin{align*}
x'_i =x_i^{-1}(\theta_i\lbrack \rho_i\rbrack(u_i,v_i)), \quad u_i := \displaystyle \prod_{j=1}^m x_j^{\lbrack b_{ji}\rbrack_+}, \quad v_i := \displaystyle \prod_{j=1}^m x_j^{\lbrack -b_{ji}\rbrack_+};
\end{align*} 

\item[(2)] $\rho':=\mu_i(\rho)=(\rho-\{\rho_i\})\cup\{\rho'_i\}$, $\rho'_i=(\rho'_{i,0}, \ldots, \rho'_{i,d_i})$, $\rho'_{i,r} = \rho_{i,d_i-r}$ for $0\leq r \leq d_i$;

\item[(3)] 
$\widetilde{B}':=\mu_i(\widetilde{B}D)D^{-1}$, $\mu_i(A)=(a'_{kj})$ is the matrix obtained from
$A=(a_{kj})$ by
\begin{gather*}
a'_{kj}=
 \begin{cases}
 -a_{kj}, &\text{if}\quad k=i\quad \text{or}\quad j=i,
 \\
a_{kj}+\displaystyle\frac{|a_{ki}|a_{ij}+a_{ki}|a_{ij}|}{2}, &\text{otherwise},
 \end{cases}
\end{gather*}
\end{itemize}
and $D$ is the diagonal matrix whose diagonal entries are $d_i$, $i\in [1, n]$.
\end{definition}

Two generalized seeds are said to be mutation-equivalent if one can be obtained from the other by a finite sequence of mutations.

\begin{definition}[{\cite[Definition 2.4]{BCDX20}}]\label{def of gca} Given an initial generalized seed $\big(\widetilde{\mathbf{x}}, \rho, \widetilde{B}\big)$, the generalized cluster algebra $\smash{\mathcal{A}\big(\widetilde{\mathbf{x}}, \rho, \widetilde{B}\big)}$ is the $\mathop{\mathbb{ZP}}$-subalgebra of $\mathcal{F}$ generated by all cluster variables from all generalized seeds which are mutation-equivalent to $\smash{\big(\widetilde{\mathbf{x}},\rho,\widetilde{B}\big)}$. The number of cluster variables in the initial cluster is called the rank of $\smash{\mathcal{A}\big(\widetilde{\mathbf{x}}, \rho, \widetilde{B}\big)}$.
\end{definition}
A generalized cluster algebra is also called a CS-cluster algebra, after Chekhov and Shapiro \cite{CS14}, see Section 2.3 of \cite{Fra20}. A generalized cluster monomial in $\smash{\mathcal{A}\big(\widetilde{\mathbf{x}}, \rho, \widetilde{B}\big)}$ is a monomial in the cluster variables involving only variables belonging to a single cluster. A generalized cluster algebra is said to be of finite type if it has finitely many cluster variables. The Laurent phenomenon from \cite[Theorem 2.5]{CS14} remains true for generalized cluster algebras.

Note that one can recover the classical cluster algebras by setting $d_i=1$ for all $i \in [1,n]$.

Throughout the remainder of this paper, we consider only generalized cluster algebras whose extended exchange matrices \(\widetilde B=(b_{ij})\) have skew-symmetric principal part \(B\). A quiver \(Q\) is an oriented graph consisting of a vertex set \(Q_0\), an arrow set \(Q_1\), and maps \(s,t:Q_1\to Q_0\) sending each arrow to its source and target, respectively.

Given an extended exchange matrix \(\widetilde B=(b_{ij})\in M_{m\times n}(\mathbb Z)\) with skew-symmetric principal part \(B\), the associated quiver \(Q=(Q_0,Q_1)\) is defined as follows. The vertex set is $Q_0=\{v_1,\dots,v_m\}$, where the vertices $v_1,\dots,v_n$ are mutable and $v_{n+1},\dots,v_m$ are frozen. If $b_{ji} > 0$, draw $b_{ji}$ arrows from $v_j$ to $v_i$; if $b_{ji} < 0$, draw $|b_{ji}|$ arrows from $v_i$ to $v_j$.  

Following \cite[Definition 9]{RW25}, we assign each mutable node $i$ of a quiver $Q$ a weight $d_i \in \ZZ_{\ge 1}$, and call $(Q, \{d_1, d_2, \ldots, d_n\})$ a node weighted quiver. We will define mutations of a node weighted quiver $(Q, \{d_1, d_2, \ldots, d_n\})$ in Definition \ref{Def: quiver mutation}. The mutation of a node weighted quiver defined in Definition \ref{Def: quiver mutation} differs slightly from that in \cite[Definition 9]{RW25}. This is mainly due to the fact that the definition of the associated extended exchange matrix in (3) of Definition \ref{Def: the mutation of a generalized cluster seed} is slightly different from that used in \cite[Definition 2]{RW25}.

\begin{definition} [{Mutation of a node weighted quiver}] \label{Def: quiver mutation}
Let $(Q, \{d_1,d_2,\ldots,d_n\})$ be a node weighted quiver with $d_i\in \mathbb{Z}_{\ge 1}$ for all $i\in[1,n]$, and let $v_i$ be a mutable vertex of $Q$. The node weighted quiver $(Q, \{d_1, d_2, \ldots, d_n\})$ mutation at the vertex $v_i$ transforms $(Q, \{d_1, d_2, \ldots, d_n\})$ into a new node weighted quiver $(\mu_i(Q), \{d_1, d_2, \ldots, d_n\} )$. The mutated quiver $\mu_i(Q)$ is defined as follows: it has the same set of vertices as $Q$, and its set of arrows is obtained by the following procedure:
 \begin{enumerate}
     \item for each subquiver $v_k\rightarrow v_i\rightarrow v_j$, add $d_i$ new arrows $v_k \rightarrow v_j$; 
     \item reverse all arrows with source or target $v_i$;
     \item remove the arrows in a maximal set of pairwise disjoint 2-cycles.
 \end{enumerate}
\end{definition}
The mutation of $(Q, \{d_1,d_2,\ldots,d_n\})$ at the vertex $v_i$ corresponds to the mutation of $\widetilde{B}$ in direction $i$ in (3) of Definition \ref{Def: the mutation of a generalized cluster seed}.

\subsection{Restricted quantum loop algebras \texorpdfstring{$U_\varepsilon^{\res}({L\mathfrak{g}})$}{ of type A} at roots of unity}
Let $\mathfrak{g}$ be a simple Lie algebra over $\mathbb{C}$ and $I$ the set of indices of the Dynkin diagram of $\mathfrak{g}$. Let $\{\alpha_i\}_{i \in I}$ be a set of simple roots and $\{\omega_i\}_{i \in I}$ the corresponding fundamental weights. Let $C=(c_{i,j})_{i,j \in I}$ be the Cartan matrix of $\mathfrak{g}$, where $c_{i,j}=\frac{2 (\alpha_i, \alpha_j) }{(\alpha_i, \alpha_i)}$. Let $\mathfrak{\widehat{g}}$ be the associated affine Lie algebra and $U_q(\widehat{\mathfrak{g}})$ the corresponding quantum affine algebra.

Denote by $U_q(L \mathfrak{g})$ the quantum loop algebra, which is isomorphic to a quotient of $U_q(\widehat{\mathfrak{g}})$ where the central charge is mapped to $1$. Frenkel and Reshetikhin \cite{FR98} introduced the theory of $q$-characters of finite-dimensional modules of $U_q(L \mathfrak{g})$. Let ${\rm Rep} U_{q}(L{\mathfrak{g}})$ be the category of finite-dimensional $U_q (L{\mathfrak{g}})$-modules and $\mathcal{K}_0({\rm Rep} U_{q}(L{\mathfrak{g}}))$ the Grothendieck ring of ${\rm Rep} U_{q}(L{\mathfrak{g}})$. The $q$-character map is an injective ring homomorphism from $\mathcal{K}_0({\rm Rep} U_{q}(L{\mathfrak{g}}))$ to $\mathbb{Z}[Y_{i,a}^{\pm1}]_{i \in I}^{a\in\mathbb{C^{\times}}}$ of Laurent polynomials in the variables $Y_{i,a}$.

Let $U_q^{\res}({L\mathfrak{g}})$ be the $\mathbb{C}[q,q^{-1}]$-subalgebra of $U_q(L\mathfrak{g})$ generated by the $q$-divided powers of the Chevalley generators {\cite[Section 1]{CP97}}. Let $\varepsilon$ be a root of unity, and define
\begin{align*}
U_\varepsilon^{\res}({L\mathfrak{g}}):=U_q^{\res}({L\mathfrak{g}}) \otimes_{\mathbb{C}[q, q^{-1}]}\mathbb{C}
\end{align*}
via the algebra homomorphism $f_{\varepsilon}: \mathbb{C}[q, q^{-1}]\rightarrow \mathbb{C}$ sending $q$ to $\varepsilon$. For an element $x$ of $U_q^{\res}({L\mathfrak{g}})$, we denote the corresponding element of $U_\varepsilon^{\res}({L\mathfrak{g}})$ also by $x$.

A finite-dimensional $U_\varepsilon^{\res}({L\mathfrak{g}})$-module $V$ has an $\varepsilon$-character $\chi_{\varepsilon}(V)$ \cite[Section 3]{FM02}, which is an element of $\mathbb{Z}[Y_{i,a}^{\pm1}]_{i\in I}^{a\in\mathbb{C^{\times}}}$. The finite-dimensional simple $U_{\varepsilon}^{\res}(L\mathfrak{g})$-modules are classified by $I$-tuples $(P_i(u))_{i \in I}$ (see \cite[Theorem 8.2]{CP97} and \cite[Theorem 2.4]{FM02}), where each $P_i(u) \in \mathbb{C}[u]$ is a polynomial with constant term $1$ which is called a Drinfeld polynomial. 

Denote by $\mathcal{P}$ the free abelian multiplicative group of monomials in infinitely many formal variables $Y_{i, a}$, $i \in I, a \in \mathbb{C}^{\times}$. A monomial $m=\prod_{i \in I, a \in \mathbb{C}^{\times}} Y_{i, a}^{u_{i, a}}$, where $u_{i, a}$ are integers, is said to be \textit{dominant} (resp. \textit{anti-dominant}) if $u_{i, a} \geq 0$ (resp. $u_{i, a} \leq 0$) for all $a \in \mathbb{C}^{\times}, i \in I$. Let $\mathcal{P}^+$ denote the set of dominant monomials. Every $I$-tuple $(P_i(u))_{i \in I}$ of Drinfeld polynomials corresponds to a dominant monomial $m \in \mathcal{P}^+$. Denote by $L(m)$ the corresponding simple $U_{\varepsilon}^{\res}(L\mathfrak{g})$-module.

Define $A_{i,a}\in \mathbb{Z}[Y_{j,b}^{\pm1}]_{j\in I}^{b\in\mathbb{C}^\times}$ by the formula
\[
A_{i,a}=Y_{i,a\varepsilon^{-1}}Y_{i,a\varepsilon}\prod_{j\neq i}Y_{j,a}^{-c_{j,i}},
\]
where $C=(c_{i,j})$ is the Cartan matrix, see \cite[Section 3.2]{FM02}. Following \cite[Equation (4.3)]{FM01}, \cite[Definition 2.9]{Nak04}, and \cite[Section 2.5]{FM02}, a partial order $\leq$ on $\mathcal{P}^+$ is defined by \[
m\le m' \text { if and only if } m'm^{-1}\in \mathcal{Q}^+.
\]
where $m, m' \in \mathcal{P}^+$, $\mathcal{Q}^{+}$ is the monoid generated by $A_{i,a}$, $i\in I$, $a\in\mathbb C^\times$.

Denote by ${\rm Rep}U_\varepsilon^{\res}({L\mathfrak{g}})$ the category of finite-dimensional modules of $U_\varepsilon^{\res}({L\mathfrak{g}})$, and by $\mathcal{K}_0({\rm Rep}U_\varepsilon^{\res}({L\mathfrak{g}}))$ its Grothendieck ring.
Frenkel and Mukhin {\cite[Theorem 3.2]{FM02}} proved that the $\varepsilon$-character map $\chi_{\varepsilon}: \mathcal{K}_0({\rm Rep} U_\varepsilon^{\res}({L\mathfrak{g}})) \rightarrow \mathbb{Z}[Y_{i,a}^{\pm1}]_{i\in I}^{a\in\mathbb{C^{\times}}}$ is an  injective ring homomorphism. 

A finite-dimensional $U_\varepsilon^{\res}(L\mathfrak{g})$-module $V$ is called \textit{special} if $\chi_{\varepsilon}(V)$ contains exactly one dominant monomial.

Since the Dynkin diagram of $\mathfrak{g}$ is a bipartite graph, we may choose a partition of the vertices $I=I_0\sqcup I_1$, where each edge connects a vertex of $I_0$ with a vertex of $I_1$. For $i\in I$, set
\begin{align}\label{Def: xi_i}
\xi_i=
\begin{cases} 0, & \text{if}{\hskip 0.4em} i\in I_0,\\
              1, & \text{if}{\hskip 0.4em} i\in I_1.
\end{cases}
\end{align}
Following \cite{FM02}, for $i\in I$ and $a\in \mathbb{C^{\times}}$, denote
\begin{align}\label{mathbf{Y}}
\mathbf{Y}_{i,a}:=\prod_{j=0}^{\ell-1} Y_{i,a\varepsilon^{2j+\xi_i}}.
\end{align}
Since $\varepsilon^{2\ell}=1$, we have $\mathbf{Y}_{i,\varepsilon^{2r}}=\mathbf{Y}_{i,1}$ for any $r\in \mathbb{Z}$. A monomial in the variables $Y_{i,a}$ is said to be $\ell$-\textit{acyclic} if it is not divisible by $\mathbf{Y}_{j,b}$ for any $j \in I, b\in \mathbb{C}^{\times}$, see \cite[Section 2.6]{FM02}

The following theorem was proved by Chari and Pressley \cite{CP97} for roots of unity of odd order and generalized by Frenkel and Mukhin \cite{FM02} to roots of unity of arbitrary order.

\begin{theorem}[{\cite[Theorem 9.1]{CP97}}, {\cite[Theorem 5.4]{FM02}}]\label{Th:decomposition}
Let $L(m)$ be a simple module of $U_\varepsilon^{\res}({L\mathfrak{g}})$. Then
\begin{align*}
L(m)\cong L(m^{0})\otimes L(m^{1}),
\end{align*}
using the decomposition $m=m^{0}m^{1}$, where $m^{1}$ is a monomial in the variables $\mathbf{Y}_{i,a}$ and $m^{0}$ is $\ell$-acyclic. Moreover, $L(m^1)$ is the Frobenius pullback of an irreducible $U_{{\varepsilon}^{\ast}}^{\res}({L\mathfrak{g}})$-module.
\end{theorem}
Note that $L(m^1)$ is the Frobenius pullback of an irreducible $U_{{\varepsilon}^{\ast}}^{\res}({L\mathfrak{g}})$-module $L(\widetilde{m}^1)$. The $\varepsilon$-character $\chi_{\varepsilon}({\rm Fr}^{\ast}(L(\widetilde{m}^1)))$ is obtained from $\chi_{\varepsilon^{\ast}}(L(\widetilde{m}^1))$ by replacing each $Y_{i,a^\ell}^{\pm 1}$ with $\mathbf{Y}_{i,a}^{\pm 1}$, where $\xi_i$ is defined in $(\ref{Def: xi_i})$, see \cite[Theorem 5.7]{FM02} and \cite[Section 3]{Gle16}. Note that the ${\bf Y}_{i,a}$ in (\ref{mathbf{Y}}) corresponds to $\mathbf{Y}_{i,a\varepsilon^{\xi_i}}$ in Section 5.3 of \cite{FM02}. 

As in \cite{FM02}, we set $\varepsilon^{*}:=\varepsilon^{\ell^2}$. By the definition of $\varepsilon$ in the introduction, $\varepsilon^{*}=\varepsilon^{\ell^2} = 1$. Therefore the category of finite-dimensional $U_{{\varepsilon}^{*}}^{\res}({L\mathfrak{g}})$-modules is equivalent to the category of finite-dimensional $L\mathfrak{g}$-modules, see \cite[Section 5.4]{FM02}. For any complex Lie algebra $\mathfrak{g}$ and a non-zero constant $b$, we have the evaluation homomorphism
\[\Phi_b: L{\mathfrak{g}}=\mathfrak{g}\otimes\mathbb{C}[t,t^{-1}]\rightarrow \mathfrak{g},\quad\,\, g\otimes t^k\mapsto b^kg.\]
For an irreducible $\mathfrak{g}$-module $V_{\lambda}$ with the highest weight $\lambda$, let $V_\lambda(b)$ be its pullback under $\Phi_b$ to an irreducible module of $L{\mathfrak{g}}$. Let $\chi(V_\lambda)$ be the ordinary character of the $\mathfrak{g}$-module $V_\lambda$, considered as a polynomial in $y_i^{\pm 1}$, $i\in I$. Then $\chi_{{\varepsilon}^\ast}(V_\lambda(b))$ is obtained from $\chi(V_\lambda)$ by replacing each $y_i^{\pm 1}$ with $Y_{i,b}^{\pm 1}$, see {\cite[Lemma 5.8]{FM02}}.

\subsection{Kirillov--Reshetikhin modules, minimal affinizations, and snake modules }\label{Snake modules} 
We first recall the definition of snake modules that was introduced by Mukhin and Young in \cite{MY12}. In this paper, we fix $a \in \CC^{\times}$ and for convenience we write $Y_{i,s}=Y_{i,a\varepsilon^s}$ for $i \in I$, $s \in \mathbb{Z}$. 

Let
\begin{align}\label{Eq:X}
\mathcal{X}:=\{(i,k)\in I \times \mathbb{Z}: i-k \equiv 1 \pmod 2\} \subset I \times \mathbb{Z}.
\end{align}

For $(i,k) \in \mathcal{X}$, a point $(i',k')$ is said to be in \textit{snake position} with respect to $(i,k)$ if
\[
k'-k \geq |i'-i|+2 \ \text{and} \  k'-k \equiv |i'-i|\pmod 2.
\]
The point $(i',k')$ is in \textit{minimal} snake position with respect to $(i,k)$ if $k'-k$ is equal to the given lower bound. 
For $(i,k) \in \mathcal{X}$, a point $(i',k') \in \mathcal{X}$ is said to be in \textit{prime snake position} with respect to $(i,k)$ if
\[
\min \{ 2n+2-i-i', i+i' \} \geq k'-k \geq |i'-i|+2 \ \text{and} \  k'-k \equiv |i'-i| \pmod 2.  \\
\]

A finite sequence $(i_{t},k_{t})$, $1 \leq t \leq z$, $z \in \mathbb{Z}_{>0}$, of points in $\mathcal{X}$ is called a \textit{snake} if for all $2 \leq t \leq z$, the point $(i_{t},k_{t})$ is in snake position with respect to $(i_{t-1},k_{t-1})$. It is called a \textit{minimal} (resp. \textit{prime}) snake if all successive points are in minimal (resp. prime) snake position \cite[Section 4]{MY12}.

The simple module $L(m)$ is called a \textit{snake module} (resp. a \textit{minimal snake module}) if $m=\prod_{t=1}^{z} Y_{i_{t},k_{t}}$ for some snake $(i_{t},k_{t})_{1 \leq t \leq z}$ (resp. for some minimal snake $(i_{t},k_{t})_{1 \leq t \leq z}$). In this case, $(i_{t},k_{t})_{1 \leq t \leq z}$ is called the snake of $L(m)$ \cite[Section 4]{MY12}.

For $(i,k)\in \mathcal{X}$, $s\in\ZZ_{\geq 1}$, define
\begin{align}\label{Eq:KR module}
X_{i,k}^{(s)}=Y_{i,k}Y_{i,k+2}\cdots Y_{i,k+2(s-1)}.
\end{align}
The modules $L(X_{i,k}^{(s)})$ are called \textit{Kirillov--Reshetikhin} modules. In particular, when $s=1$, the modules $L(X_{i,k}^{(1)})=L(Y_{i,k})$ are called \textit{fundamental modules}.

For the precise definition of minimal affinizations, we refer the reader to \cite{C95, CP96}. The following is an equivalent definition of minimal affinizations.
\begin{proposition}[{\cite[Proposition 4.2]{MY12}}]\label{Prop:minimal affine module}
Let $(i_t,k_t) \in \mathcal{X}$, $1\leq t\leq T$, $T\in\ZZ_{\geq 1}$. The module $L(\prod^T_{t=1}Y_{i_t,k_t})$ is a minimal affinization if and only if $(i_t,k_t)_{1\leq t\leq T}$ is a minimal snake and the sequence $(i_t)_{1\leq t\leq T}$ is monotonic.
\end{proposition}

A \textit{path} is a finite sequence of points in the plane $\mathbb{R}^{2}$. We write $(j,l) \in p$ if $(j,l)$ is a point of the path $p$. When we draw paths, we connect consecutive points of the path by line segments, for illustrative purposes, see Figure \ref{F: figure 1}.

For $(i,k) \in \mathcal{X}$, let
\begin{align}\label{path}
\mathscr{P}_{i,k}=\{ & ((0,y_{0}),(1,y_{1}),\ldots,(n+1,y_{n+1})): y_{0}=i+k, \nonumber \\
&y_{n+1}=n+1-i+k, \text{ and } y_{j+1}-y_{j}\in \{1,-1\}, \  0\leq j\leq n\}.
\end{align}

The sets $C_{p}^{\pm}$ of upper and lower corners of a path $p=((r,y_{r}))_{0\leq r \leq n+1}\in \mathscr{P}_{i,k}$ are defined as follows:
\begin{align*}
C^{+}_{p}=\{(r,y_{r})\in p: r\in I, \ y_{r-1}=y_{r}+1=y_{r+1}\},\\
C^{-}_{p}=\{(r,y_{r})\in p: r\in I, \ y_{r-1}=y_{r}-1=y_{r+1}\}.
\end{align*}

A map $\mathfrak{m}$ sending paths to monomials is defined by
\begin{align}\label{map: path to monomial}
\mathfrak{m}: \bigsqcup_{(i,k)\in \mathcal{X}}{\hskip -0.5em}\mathscr{P}_{i,k} & \longrightarrow \mathbb{Z}[Y_{j,l}^{\pm 1}]_{(j,l)\in \mathcal{X}} \nonumber \\
p\quad & \longmapsto  \mathfrak{m}(p)=\prod_{(j,l)\in C^{+}_{p}}{\hskip -0.5em}Y_{j,l}{\hskip -0.5em}\prod_{(j,l)\in C^{-}_{p}}{\hskip -0.5em}Y_{j,l}^{-1}.
\end{align}

Let $p, p'$ be paths. The path $p$ is \textit{strictly above} $p'$ or $p'$ is \textit{strictly below} $p$ if
\begin{align*}
(x,y)\in p \text{ and } (x,z)\in p' \Longrightarrow y < z.
\end{align*}
A $z$-tuple of paths $(p_{1},\ldots,p_{z})$ is called \textit{non-overlapping} if $p_{s}$ is strictly above $p_{t}$ for all $s<t$. For any snake $(i_{t},k_{t})\in \mathcal{X}$, $1\leq t\leq z$, $z\in \mathbb{Z}_{\geq1}$, let
\begin{align}\label{Non-overlapping paths}
\overline{\mathscr{P}}_{(i_{t},k_{t})_{1\leq t\leq z}}=\{(p_{1},\ldots,p_{z}): p_{t}\in \mathscr{P}_{i_{t},k_{t}}, 1\leq t\leq z, (p_{1},\ldots,p_{z})\text{ is } \text {non-overlapping}\}.
\end{align}

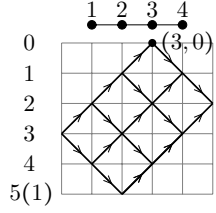
\begin{figure}
\resizebox{0.8\width}{0.8\height}{
\begin{minipage}[b]{0.4\linewidth}
\centerline{
\begin{tikzpicture}
\draw[step=.5cm,gray,thin] (-0.5,5.5) grid (2,8) (-0.5,5.5)--(2,5.5);
\draw[fill] (0,8.3) circle (2pt) -- (0.5,8.3) circle (2pt) --(1,8.3) circle (2pt) --(1.5,8.3) circle (2pt);
\begin{scope}[thick, every node/.style={sloped,allow upside down}]
\draw (-0.5,6.5)--node {\midarrow}(0,6);
\draw (0,6)--node {\midarrow}(0.5,5.5);
\draw (0.5,5.5)--node {\midarrow}(1,6);
\draw (1,6)--node {\midarrow}(1.5,6.5);
\draw (1.5,6.5)--node {\midarrow}(2,7);
\draw (-0.5,6.5)--node {\midarrow}(0,7);
\draw (0,7)--node {\midarrow}(0.5,7.5);
\draw (0.5,7.5)--node {\midarrow}(1,8);
\draw (1,8)--node {\midarrow}(1.5,7.5);
\draw (1.5,7.5)--node {\midarrow}(2,7);
\draw (0,7)--node {\midarrow}(0.5,6.5);
\draw (0.5,6.5)--node {\midarrow}(1,6);
\draw (0.5,7.5)--node {\midarrow}(1,7);
\draw (1,7)--node {\midarrow}(1.5,6.5);
\draw (1,7)--node {\midarrow}(1.5,7.5);
\draw (0,6)--node {\midarrow}(0.5,6.5);
\draw (0.5,6.5)--node {\midarrow}(1,7);
\draw (0.5,5.5)--node {\midarrow}(1,6);
\draw (1,6)--node {\midarrow}(1.5,6.5);
\draw[fill] (1,8) circle (1.5pt);
\end{scope}
\node [above] at (0,8.3)   {$1$};
\node [above] at (0.5,8.3) {$2$};
\node [above] at (1,8.3)   {$3$};
\node [above] at (1.5,8.3) {$4$};
\node [left] at (-0.8,8)   {$0$};
\node [left] at (-0.8,7.5) {$1$};
\node [left] at (-0.8,7)   {$2$};
\node [left] at (-0.8,6.5) {$3$};
\node [left] at (-0.8,6)   {$4$};
\node [left] at (-0.5,5.5) {$5(1)$};
\node[right] at (1,8)    {$(3,0)$};
\end{tikzpicture}}
\end{minipage}}
\caption{Let $\varepsilon^{2 \ell}=1$ with $\ell=2$. The paths corresponding to the monomials of $U_{\varepsilon}^{\res}({L\mathfrak{sl}_5})$-module $\chi_\varepsilon(L(Y_{3,0}))$.}\label{F: figure 1}
\end{figure}

For $(i,k) \in \mathcal{X}$, the highest path in $\mathscr{P}_{i,k}$ is the unique path with no lower corners, and the lowest path in $\mathscr{P}_{i,k}$ is the unique path with no upper corners.

When we draw paths, we relabel the Row $r$ below $x=2 \ell$ by the remainder of $r$ modulo $2 \ell$, see for example, Figure \ref{F: figure 1} where Row $5$ is labeled by $1$.

It was observed in Theorem 3.1 in \cite{ALLZ25} that every simple $U_{\varepsilon}^{\res}({L\mathfrak{sl}_k})$-module $L(M)$ can be written as a snake module $L(M')$, where $M'=M$ in $\mathcal{P}_{\varepsilon, \xi}^+$. The $q$-character of a snake module can be computed using Mukhin--Young's path formula in \cite{MY12}. Suppose that $M$ is $\ell$-acyclic (see Theorem \ref{Th:decomposition}). After specializing $q=\varepsilon$, the resulting expression $\chi_q(L(M'))|_{q=\varepsilon}$ may contain additional dominant monomials, reflecting contributions from lower composition factors of the specialized module. By \cite[Proposition 2.5]{FM02}, we have 
\begin{align*}
\chi_q(L(M'))\big|_{q=\varepsilon}
= \chi_{\varepsilon}(L(M'))+\sum_{m < M'} a_m \chi_{\varepsilon}(L(m)),
\qquad a_m \in \mathbb{Z}_{\ge 0}.
\end{align*}
The $\varepsilon$-character $\chi_{\varepsilon}(L(M))=\chi_{\varepsilon}(L(M'))$ can be obtained by subtracting the $\varepsilon$-characters of the lower simple modules in the above formula.
 
In particular, in Theorem 5.2 in \cite{ALLZ25}, an explicit path description of $\varepsilon$-character of the Kirillov--Reshetikhin $U_{\varepsilon}^{\res}({L\mathfrak{sl}_k})$-module $(L(X_{i,t}^{(s)}))$, $s\leq \ell$, is given. More precisely, 
\[
\chi_\varepsilon(L(X_{i,t}^{(s)}))=\sum_{(p_{1},\ldots,p_{s}) \in \overline{\mathscr{P'}}_{(i,t+2r)_{0 \leq r \leq s-1}}} \prod_{j=1}^{s}\mathfrak{m}(p_{j}),
\]
where $\mathfrak{m}(p_j)$ denotes the monomial associated with the path $p_j$ (see (\ref{map: path to monomial})) and 
\[
\overline{\mathscr{P'}}_{(i,t+2r)_{0\le r\le s-1}}
:=\left\{(p_1,\ldots,p_s)\ \middle|\
\begin{array}{l}
p_j \in \mathscr{P}_{i,t+2(j-1)},\ 
p_u \cap p_v = \emptyset\ \text{for all } u\ne v
\end{array}
\right\}.
\]
Here $\mathscr{P}_{i,t+2(j-1)}$ denotes the set of paths associated with $(i,t+2(j-1))$ (see (\ref{path})) in a tube (i.e., the lattice with the identification $t \sim t+2\ell$), and the condition $p_u \cap p_v = \emptyset$ means that the paths are pairwise disjoint.

\section{Fraser's conjecture and Gleitz's conjecture}\label{Section: Fraser's conjecture and Gleitz's conjecture}
In this section, we recall cyclic symmetry loci in Grassmannians, Fraser's conjecture \cite{Fra20}, and Gleitz's conjecture \cite{Gle16}. We prove that the monoid $\mathcal{P}^+_{\varepsilon, \xi}$ is isomorphic to the monoid ${\rm SSYT}(k, [k+\ell], \sim)$ of equivalence classes of semistandard tableaux. Using this isomorphism, we reformulate Fraser's conjecture in the language of quantum affine algebras.

\subsection{Parametrization of simple modules using semistandard Young tableaux}\label{Sec:the connection between semistandard tableaux and dominant monomials}

By definition (see Section \ref{subsec:introduction:a proof of Gleitz conjecture}), the monoid $\mathcal{P}^+_{\varepsilon,\xi}$ is generated by the distinct elements $\{Y_{i,s} \mid i \in I_0,\, s = 0,2,\dots,2\ell-2\} \cup \{Y_{i,s} \mid i \in I_1,\, s = 1,3,\dots,2\ell-1\}$, 
and there are no relations among them except commutativity, so $\mathcal{P}^+_{\varepsilon,\xi}$ is a free commutative monoid.

Denote by $\mathcal{P}_{\ell}^+$ the submonoid of $\mathcal{P}^+$ generated by $Y_{i,i-2r-2}$, $i \in I$, $r \in [0, \ell]$.  
\begin{lemma} \label{lem:Pepsilon_xi is isomorphic to Pell_minus_one}
Suppose that $\varepsilon^{2\ell}=1$. Then the monoid $\mathcal{P}^+_{\varepsilon, \xi}$ is isomorphic to $\mathcal{P}_{\ell-1}^+$.   
\end{lemma}

\begin{proof}
Define a map $\Psi:\ \mathcal{P}_{\ell-1}^+ \longrightarrow \mathcal{P}^+_{\varepsilon,\xi}$
by
\[
\Psi(Y_{i,s}) := Y_{i, \overline{s-1}},
\]
where $\overline{\cdot}$ denotes the representative in $[0,2\ell-1]$. For each fixed $i$, this map induces a bijection between the set
$\{i-2r-2 \mid r=0,\dots,\ell-1\}$
and the set $\{s \in [0,2\ell-1] \mid s \equiv \xi_i \pmod{2}\}$.
Since $\mathcal{P}_{\ell-1}^+$ is the free commutative monoid generated by $Y_{i,i-2r-2}$, $i \in I$, $r \in [0, \ell-1]$, the map $\Psi$ extends uniquely to a monoid homomorphism. By construction, $\Psi$ maps the generators bijectively onto the generators of $\mathcal{P}^+_{\varepsilon,\xi}$, hence it is surjective. Moreover, since both monoids are free commutative monoids and $\Psi$ sends a basis to a basis bijectively, distinct monomials in $\mathcal{P}_{\ell-1}^+$ have distinct images, so $\Psi$ is injective. Therefore, $\Psi$ is an isomorphism of commutative monoids.
\end{proof}

A semistandard Young tableau is a Young tableau with weakly increasing rows and strictly increasing columns. For $k,n \in \ZZ_{\ge 1}$, we denote by ${\rm SSYT}(k, [n])$ the set of rectangular semistandard Young tableaux with $k$ rows and with entries in $[n]$ (with arbitrarly many columns). The empty tableau is denoted by $\mathds{1}$. For $S,T \in {\rm SSYT}(k, [n])$, let $S \cup T$ be the row-increasing tableau whose $i$th row is the union of the $i$th rows of $S$ and $T$ (as multisets), \cite{CDFL20}. It is shown in Section 3 in \cite{CDFL20} that $S \cup T$ is semistandard for any pair of semistandard tableaux $S, T$.

A tableau $T \in {\rm SSYT}(k, [n])$ is trivial if each entry of $T$ is one less than the entry below it. For any $T \in {\rm SSYT}(k, [n])$, we denote by $T_{\text{red}} \subset T$ the semistandard tableau obtained by removing a maximal trivial factor from $T$. For a trivial $T$, one has $T_{\text{red}} = \mathds{1}$. Let ``$\sim$'' be the equivalence relation on $S, T \in {\rm SSYT}(k, [n])$ defined by: $S \sim T$ if and only if $S_{\text{red}} = T_{\text{red}}$. We denote by ${\rm SSYT}(k, [n],\sim)$ the set of $\sim$-equivalence classes.
A one-column tableau is called a fundamental tableau if its content is $[i,i+k] \setminus \{r\}$ for $r \in \{i+1, \ldots, i+k-1\}$.
Any tableau in $\SSYT(k,[n])$ is $\sim$-equivalent to a unique semistandard tableau whose columns are fundamental tableaux, see Lemma 3.13 in \cite{CDFL20}. 

Let $k \ge 1$. We denote by ${\rm SSYT}(k, \ZZ)$ the set of rectangular semistandard Young tableaux with $k$ rows and entries in $\ZZ$ (allowing arbitrarily many columns). Let ${\rm SSYT}(k, \ZZ, \sim)$ be the monoid of $\sim$-equivalence classes in ${\rm SSYT}(k, \ZZ)$.
Let $\varepsilon^{2\ell} = 1$ with $\ell \in \mathbb{Z}_{\ge 2}$. Define ${\rm SSYT}(k, \mathbb{Z},  \ell, \sim)$ as the quotient of ${\rm SSYT}(k, \mathbb{Z}, \sim)$ by identifying all one-column tableaux of the form $[i, i+1, \dots, i+k]^\top \setminus \{r\}$ and $[i+\ell, i+\ell+1, \dots, i+\ell+k]^\top \setminus \{r+\ell\}$ for any $i \in \ZZ$ and $r \in [i, i+k]$. Similarly, define ${\rm SSYT}(k, \mathbb{Z},   \ell)$ as the quotient of ${\rm SSYT}(k, \mathbb{Z})$ by identifying all one-column tableaux of the form $[i, i+1, \dots, i+k]^\top \setminus \{r\}$ and $[i+\ell, i+\ell+1, \dots, i+\ell+k]^\top \setminus \{r+\ell\}$ for any $i \in \ZZ$ and $r \in [i, i+k]$. The following lemma is immediate. 

\begin{lemma}  \label{lem:quotient of SSYTkZ is isomorphic to SSYTk_ell}
There are natural monoid isomorphisms 
\begin{align*}
& {\rm SSYT}(k, \mathbb{Z},  \ell, \sim) \;\cong\; {\rm SSYT}(k, [k+\ell], \sim), \\
& {\rm SSYT}(k, \mathbb{Z},  \ell) \;\cong\; {\rm SSYT}(k, [k+\ell])
\end{align*} 
given by mapping each column of a tableau to its representative in $[1, k+\ell]$. 
\end{lemma}

By Theorem 3.17 in \cite{CDFL20}, the monoids $\mathcal{P}_{\ell-1}^+$ and $\SSYT(k, [k+\ell], \sim)$ are isomorphic. Combining this with Lemmas \ref{lem:Pepsilon_xi is isomorphic to Pell_minus_one} and \ref{lem:quotient of SSYTkZ is isomorphic to SSYTk_ell}, we obtain the following theorem.

\begin{theorem}\label{thm: isomorphic tableaux}
Let $\varepsilon^{2 \ell}=1$ with $\ell \in \ZZ_{\geq 2}$. The monoid $\mathcal{P}^+_{\varepsilon, \xi}$ of dominant monomials is isomorphic to the monoid $\SSYT(k, [k+\ell], \sim)$ of semistandard Young tableaux.
\end{theorem}

In the following, we describe the isomorphism $\Phi: \mathcal{P}^+_{\varepsilon,\xi} \to \SSYT(k, [k+\ell], \sim)$ in Theorem \ref{thm: isomorphic tableaux} explicitly. 

For $k \in \mathbb{Z}_{\ge 2}$, $\ell \in \mathbb{Z}_{\ge 2}$, and $x \in \mathbb{Z}$, denote by $\bar{x}$ the representative of $x \bmod \ell$ taken in $[1,\ell]$. For $i \in I$, $s \in [0,2\ell-1]$, we denote $T_{i,s}$ by the one column tableau whose entries are $[\bar{x}, \bar{x}+k] \setminus \{\bar{x}+k-i\}$, where $x=\frac{i-s-1}{2}$. The one-column tableau $T_{i,s}$ is a fundamental tableau. 

Every element in $\mathcal{P}^+_{\varepsilon,\xi}$ is of the form $m = \prod_{i,s} Y_{i,s}^{a_{i,s}}$ for some $a_{i,s} \ge 0$. The isomorphism $\Phi: \mathcal{P}^+_{\varepsilon,\xi} \to \SSYT(k, [k+\ell], \sim)$ sends $m = \prod_{i,s} Y_{i,s}^{a_{i,s}}$ to $\bigcup_{i,s} T_{i,s}^{\cup a_{i,s}}$. 

On the other hand, every tableau in $\SSYT(k, [k+\ell], \sim)$ admits a unique factorization as a $\cup$-product of fundamental tableaux $T_{i,s}$, each appearing with multiplicity $a_{i,s} \ge 0$. The inverse map $\Phi^{-1}$ then sends $T$ to $\prod_{i,s} Y_{i,s}^{a_{i,s}}$.

In Table \ref{table:correspondence between tableaux and simple modules Gr36}, we give the correspondence between one-column tableaux in $\SSYT(3, [6], \sim)$ and dominant monomials.

\begin{table}[H]
\begin{equation*}
\hspace{2cm}
\begin{minipage}{0.6\textwidth}
\begin{tabular}{|c|c|}
\hline
dominant monomials        & tableaux  \\
\hline
 $L(Y_{1,0})$        & $[3, 4, 6]$   \\
\hline 
 $L(Y_{1,2})$        & $[2, 3, 5]$   \\
\hline
 $L(Y_{1,4})$        & $[1, 2, 4]$   \\
\hline
 $L(Y_{2,1})$        & $[3, 5, 6]$   \\
\hline
 $L(Y_{2,3})$        & $[2, 4, 5]$  \\
\hline
 $L(Y_{2,5})$        & $[1, 3, 4]$   \\
\hline
 $L(Y_{1,2}Y_{1,4})$     & $[1, 2, 5]$  \\
\hline
 $L(Y_{1,0}Y_{1,2}Y_{1,4})$  & $[1, 2, 6]$  \\
\hline
\end{tabular}
\end{minipage}
\hspace{-3cm}
\begin{minipage}{0.8\textwidth}
\begin{tabular}{|c|c|}
\hline
dominant monomials        & tableaux  \\
\hline
 $L(Y_{1,2}Y_{2,5})$       & $[1, 3, 5]$  \\ 
 \hline
 $L(Y_{1,0}Y_{1,2}Y_{2,5})$    & $[1, 3, 6]$  \\ 
\hline
$L(Y_{2,5}Y_{2,3})$        & $[1, 4, 5]$  \\
 \hline
 $L(Y_{1,0}Y_{2,3}Y_{2,5})$    & $[1, 4, 6]$  \\
 \hline
 $L(Y_{2,1}Y_{2,3}Y_{2,5})$    & $[1, 5, 6]$  \\
\hline
 $L(Y_{1,0}Y_{1,2})$       & $[2, 3, 6]$  \\
\hline
 $L(Y_{1,0}Y_{2,3})$       & $[2, 4, 6]$   \\
\hline
 $L(Y_{2,1}Y_{2,3})$       & $[2, 5, 6]$  \\
\hline
\end{tabular}
\end{minipage}
\end{equation*}
\caption{Correspondence between one-column tableaux in ${\rm SSYT}(3,[6], \sim)$ and dominant monomials in $\mathcal{P}^+_{\varepsilon, \xi}$, $\varepsilon^ {2\ell}=1$, $\ell=3$. Here we write a one-column tableau as a row to save space. For example, $[3,4,6]$ is the one-column tableau with entries $3,4,6$.} \label{table:correspondence between tableaux and simple modules Gr36}
\end{table}

\begin{example}
When computing the monomial corresponding to a given tableau, we first decompose the tableau into a union of fundamental tableaux. Then we send each fundamental tableau to the corresponding $Y_{i,s}$. For example, for $k=3, \ell=3$, we have that 
\begin{align*}
\vcenter{\hbox{
\begin{ytableau}
1 & 2 \\
3 & 4 \\
5 & 6
\end{ytableau}
}}
\,\sim \,
\vcenter{\hbox{
\begin{ytableau}
1 & 2 & 2 & 3   \\
3 & 3 & 4 & 4 \\
4 & 5 & 5 & 6
\end{ytableau}
}}.
\end{align*}
Each column of the tableau on the right-hand side corresponds to one $Y_{i,s}$, and the tableau 
$\vcenter{\hbox{\scalebox{0.6}{$\begin{ytableau}
1 & 2 \\
3 & 4 \\
5 & 6
\end{ytableau}$}}}$
corresponds to the dominant monomial $Y_{2,5}Y_{1,2}Y_{2,3}Y_{1,0}$.
\end{example}

Simple modules in the category $\mathcal{C}_{\varepsilon, \xi}$ are parametrized by dominant monomials in $\mathcal{P}^+_{\varepsilon,\xi}$. Following Theorem \ref{thm: isomorphic tableaux}, we obtain that simple modules in $\mathcal{C}_{\varepsilon, \xi}$ are also parametrized by semistandard tableaux in $\SSYT(k,[k+\ell],\sim)$. Pl\"{u}cker coordinates are one-column tableaux. In particular, we have correspondence between Pl\"{u}cker coordinates and dominant monomials. This will be used in the next subsection to rewrite Fraser's conjecture.

\subsection{Cyclic symmetry loci in Grassmannians, Fraser's conjecture, and Gleitz's conjecture}

In this subsection, we recall Fraser's conjecture. As mentioned in the last sentence of Section 8 in \cite{Fra20}, the integer ``$n$'' in $\Gr(k,n)$ plays no role in Fraser's Conjecture 8.8 in \cite{Fra20}, we present Fraser's conjecture using the infinite Grassmannian $\Gr(k, \infty):= \varinjlim_{n} \Gr(k,n)$, where the limit is taken with respect to the natural inclusions
\[
\Gr(k,n) \hookrightarrow \Gr(k,n+1),  \quad V \mapsto V \oplus 0,
\]
see \cite{Sato89, SW85}. Equivalently, $\Gr(k,\infty)$ parametrizes $k$-dimensional subspaces of
\[
\CC^\infty := \bigoplus_{i\in \ZZ} \CC e_i,
\]
spanned by vectors of finite support, i.e., vectors $v = \sum_{i\in \ZZ} a_i e_i$ with $a_i=0$ for all but finitely many $i$.
The infinite Grassmannian $\Gr(k,\infty)$ admits Plücker coordinates indexed by $k$-element subsets $I \subset \mathbb{Z}$ with finite support:
\[
X \longmapsto [P_I(X)] \in \mathbb{P}(\Lambda^k \CC^\infty).
\]
Denote by $\boldsymbol{\rho}$ the linear operator
\[
\boldsymbol{\rho} : \CC^\infty \to \CC^\infty, \quad \boldsymbol{\rho}(e_i) = e_{i+1}.
\]
For $\ell \ge 1$, we write $\boldsymbol{\rho}^\ell(e_i) = e_{i+\ell}$.
This induces an action of $\boldsymbol{\rho}^\ell$ on $\Gr(k,\infty)$ by
\[
\boldsymbol{\rho}^\ell(X) := \{ \boldsymbol{\rho}^\ell(v) \mid v \in X \}, \quad X \in \Gr(k,\infty).
\]
In Plücker coordinates, this action is given by $P_I(\boldsymbol{\rho}^\ell X) = P_{I+\ell}(X)$,
where for a $k$-subset $I=\{i_1<\cdots<i_k\}$ we set $I+\ell := \{i_1+\ell,\dots,i_k+\ell\}$. We say that $X \in \Gr(k,\infty)$ is $\ell$-fixed if $\boldsymbol{\rho}^\ell(X) = X$. 
Denote the $\ell$-fixed locus of $\Gr(k,\infty)$ by
\[
\Gr(k,\infty)^{\boldsymbol{\rho}^\ell} := \left\{ X \in \Gr(k,\infty) \;\middle|\; \boldsymbol{\rho}^\ell(X)=X \right\}.
\]
Equivalently, in Plücker coordinates, the condition $\boldsymbol{\rho}^\ell(X) = X$ is
\[
[P_{I+\ell}(X)] = [P_I(X)]
\quad \text{in } \mathbb{P}(\Lambda^k \mathbb{C}^\infty),
\]
which, by homogeneity, is equivalent to the existence of $\zeta \in \mathbb{C}^{\times}$ such that
\[
P_{I+\ell}(X) = \zeta\, P_I(X) \quad \text{for all } I \in \binom{\mathbb{Z}}{k} = \{I \subset \ZZ: |I|=k\}.
\]
Thus $\Gr(k,\infty)^{\boldsymbol{\rho}^\ell}$ decomposes into components indexed by $\zeta \in \mathbb{C}^{\times}$:
\[
\Gr(k,\infty)^{\boldsymbol{\rho}^\ell} =\bigsqcup_{\zeta \in \mathbb{C}^{\times}}\Gr(k,\infty)^{\boldsymbol{\rho}^\ell,\zeta},
\]
where
\[
\Gr(k,\infty)^{\boldsymbol{\rho}^\ell,\zeta}:=\left\{X \in \Gr(k,\infty)\;\middle|\;P_{I+\ell}(X) = \zeta\, P_I(X)\text{ for all } I\right\}.
\]
Denote by $\mathcal{D}(k,\ell):=\Gr(k,\infty)^{\boldsymbol{\rho}^\ell, 1}$ the component corresponding to $\zeta=1$. Equivalently,
\[
\mathcal{D}(k,\ell)=\left\{X \in \Gr(k,\infty)\;\middle|\;P_{I+\ell}(X) = P_I(X)\text{ for all } I\right\}.
\]
Its coordinate ring is defined as the quotient of the polynomial ring generated by all Plücker coordinates $P_I$, where $I \subset \mathbb{Z}$ and $|I|=k$, by the ideal generated by the Plücker relations together with the periodicity relations $P_{I+\ell} = P_I$ for all $I$. We define $\mathbb{C}[\mathcal{D}(k,\ell,\sim)]$ to be the quotient of $\mathbb{C}[\mathcal{D}(k,\ell)]$ by the ideal generated by $P_{i,i+1,\ldots,i+k-1}-1$, $i \in \ZZ$.

Fraser constructed a generalized cluster algebra structure on $\mathbb{C}[\mathcal{D}(k, \ell)]$ which induces a natural generalized cluster algebra structure on $\mathbb{C}[\mathcal{D}(k, \ell, \sim)]$. As mentioned in the last sentence of Section 8 in \cite{Fra20}, ``$n$'' in $\Gr(k,n)$ plays no role in Conjecture 8.8 in \cite{Fra20}, we modify Definition 7.1 in \cite{Fra20} accordingly in the following. 
\begin{definition}[{\cite[Definition 7.1]{Fra20}}] \label{defn:initPluckers}
Given $k, \ell\geq 2$, set $N := (k-1)(\ell-1)$. Define a sequence $(I_i)_{i=1}^N$ as follows. 
For $j \in [k-1]$ and $a \in [\ell-1]$, set $i = (j-1)(\ell-1) + a$ and define 
\begin{equation}\label{eq:Ilist}
I_i = \{i, i+1, \dots, i+j-1\} \cup \{1+j\ell,1+(j+1)\ell,\dots,1+(k-1)\ell\}.
\end{equation}
In other words, $I_i = \{i,\ell+1,2\ell+1,\dots,(k-1)\ell+1\}$ for $i \in [\ell-1]$, $I_i = \{i,i+1,2\ell+1,\dots,(k-1)\ell+1\}$ for $i \in [\ell,2(\ell-1)]$, and so on. 
Define also $I_{N+t} = [N+t,N+t+k-1] $ for $t \in [\ell]$. 
\end{definition}

\begin{definition}[{\cite[Definition 7.2]{Fra20}}]\label{defn:abstract CS seed}
For arbitrary $k,\ell \geq 2$ and $N = (k-1)(\ell-1)$, the generalized seed $(\widetilde{\bf x},\tilde{B},{\bf z})$ is defined as follows. The initial cluster $\widetilde{\bf x} = (x_i)_{1 \le i \le N+\ell}$ consists of independent indeterminates, where the last $\ell$ variables $x_{N+1},\dots,x_{N+\ell}$ are frozen. The extended exchange matrix $\tilde{B}$ is defined by an extended quiver with arrows $x_{i+1} \to x_i \to x_{i+\ell} \to x_{i+1}$ for $i=1,\ldots,N$. The exchange degrees are $d_1 = k$ and $d_i = 1$ for $i \in [2,N]$. The coefficient string variables are abbreviated by setting $z_s := z_{1;s}$. Denote the resulting generalized cluster algebra by $\mathcal{A}_{\rm cyc}(k,\ell)$. 
\end{definition}

Let $q$ be an indeterminate. For $k,m \in \mathbb{N}$, set 
\begin{equation*}
[k]_q = \frac{q^{\frac{k}{2}}-q^{-\frac{k}{2}}}{q^{\frac{1}{2}}-q^{-\frac{1}{2}}}, \hspace{.6cm}  [m]!_q = \prod_{k=1}^m [k]_q, \hspace{.6cm} 
\genfrac{[}{]}{0pt}{}{m}{k}_q = \frac{[m]!_q}{[k]!_q [m-k]!_q}, 
\end{equation*}
the $q$-analogs of $k \in \mathbb{N}$, $m!$, and $\binom {m}{k}$, respectively. Let $p \ge k$ be an integer. Define
\begin{equation}\label{eq:coeffs}
\eta_{s} := \genfrac{[}{]}{0pt}{}{k}{s}_q \text{ evaluated at } q = e^{\frac{2\pi i}{p}}.
\end{equation}
Denote by $\mathcal{A}(k,\ell)$ the subalgebra of $ \mathbb{C}(\mathcal{D}(k,\ell))$ obtained by identifying  $x_i \in \mathcal{A}_{\rm cyc}(k,\ell)$ with $P_{I_i}\in\mathbb{C}[\mathcal{D}(k,\ell)]$, and by specializing $z_s \mapsto \eta_s$.

Let $\mathcal{A}= \smash{\mathcal{A}\big(\widetilde{\mathbf{x}}, \rho, \widetilde{B}\big)}$ be a generalized cluster algebra. Its \emph{upper generalized cluster algebra} \cite{GSV18} is defined by
\begin{align}\label{Eq: the definition of upper generalized cluster algebra}
\mathcal{A}^{\mathrm{up}} =\bigcap_{t} \mathbb{Z}\mathbb{P}\big[\mathbf{x}_t^{\pm1}\big],
\end{align}
where the intersection is taken over all seeds $t$ that are mutation-equivalent to 
$\big(\widetilde{\mathbf{x}}, \rho, \widetilde{B}\big)$, and $\mathbf{x}_t$ denotes the cluster at $t$.

Fraser conjectured that $\mathcal{A}(k,\ell)$ is strictly contained in $\mathbb{C}[\mathcal{D}(k, \ell)]$ and $\mathbb{C}[\mathcal{D}(k, \ell)]$ is equal to the upper cluster algebra $\mathcal{A}^{\rm up}(k,\ell)$ of $\mathcal{A}(k,\ell)$, see Conjecture 7.6 and the discussion below it in \cite{Fra20}.

We rewrite Fraser's Conjecture 8.8 in \cite{Fra20} in the language of quantum affine algebras as follows. 
\begin{conjecture}[{\cite[Conjecture 8.8]{Fra20}}]\label{Conj: rewrite Fraser's conjecture}
The Grothendieck ring $K_0(\mathcal{C}_{\varepsilon, \xi})$ admits an upper generalized cluster algebra structure in which each cluster monomial is the class of a simple module. The exchange degrees and the initial quiver of an initial seed of this generalized cluster algebra are  given in Definition \ref{defn:abstract CS seed}. The initial cluster in this initial seed is
\[
\mathbf{x}=\{x_1,x_2,\ldots,x_{(k-1)(\ell-1)},\,x_{(k-1)(\ell-1)+1},\ldots,x_{(k-1)(\ell-1)+\ell}\},
\]
where the variables $x_i$ for $i \in [1,(k-1)(\ell-1)+\ell]$ are given as follows. For each $i \in [1,(k-1)(\ell-1)]$, the dominant monomial of $x_i$ is expressed as a product of factors in lexicographic order. Explicitly, 
\begin{align*}
x_1 &=[ L(Y_{1,2k-2} Y_{1,2k} \cdots Y_{1,2k+2\ell-6} \, 
Y_{2,2k+2\ell-3} Y_{2,2k+2\ell-1} \cdots Y_{2,2k+4\ell-7}\times \\
&\quad \times Y_{3,2k+4\ell-4} \cdots Y_{3,2k+6\ell-8} \cdots 
Y_{k-1,2k+2(k-2)\ell-k} \cdots Y_{k-1,2k+2(k-1)\ell-(k+4)})],
\end{align*}
and for each $i \in [1,(k-1)(\ell-1)]$, $x_{i+1}$ is obtained by removing the last factor of the dominant monomial of $x_i$.  
Finally, the remaining elements are frozen variables:
\[
x_{\,(k-1)(\ell-1)+1} = \cdots = x_{\,(k-1)(\ell-1)+\ell} = 1.
\]
\end{conjecture}

According to Conjectures 7.6 and 8.8 in \cite{Fra20}, we expect that there is an isomorphism $\widetilde{\Phi}: K_0(\mathcal{C}_{\varepsilon, \xi}) \to \mathbb{C}[\mathcal{D}(k,\ell,\sim)]$ which satisfies: $\chi_{\varepsilon}(L(Y_{i,s}))$ is sent to the Pl\"{u}cker coordinate which has indices that are the entries of the one column tableau $T_{i,s}$.

\begin{example} 
Let $\varepsilon^{2 \ell}=1$ with $\ell =3$. An initial cluster for the generalized cluster algebra $\mathcal{K}_0(\Rep U_{\varepsilon}^{\res}({L\mathfrak{sl}_{4}}))$ is as follows. The initial cluster is $\mathbf{x}=\{x_1, x_2, \ldots, x_9\}$, where the cluster variables are given by
\begin{align*}
&x_1=[L(Y_{1,0} Y_{1,2} Y_{2,5}Y_{2,7} Y_{3,10} Y_{3,12})],\, x_2=[L(Y_{1,0} Y_{1,2} Y_{2,5}Y_{2,7} Y_{3,10})], x_3=[L(Y_{1,0} Y_{1,2} Y_{2,5}Y_{2,7})],\\
&x_4=[L(Y_{1,0} Y_{1,2} Y_{2,5})],\, x_5=[L(Y_{1,0} Y_{1,2})],\,x_6=[L(Y_{1,0})], \quad x_7=x_8=x_9=1.
\end{align*}
The extended exchange matrix is
\begin{equation}\label{extend exchhange matrix}
\tilde{B}=
\left(
  \begin{array}{cccccc}
    0 & -1 & 0 & 1 & 0 & 0 \\
    1 & 0 & -1 & -1 & 1 & 0 \\
    0 & 1 & 0 & -1 & -1 & 1\\
    -1 & 1 & 1 & 0 & -1 & -1 \\
    0 & -1 & 1 & 1 & 0 & -1\\
    0 & 0 & -1 & 1 & 1 & 0 \\
    0 & 0 & 0 & -1 & 1 & 1\\
    0 & 0 & 0 & 0 & -1 & 1 \\
    0 & 0 & 0 & 0 & 0 & -1\\
  \end{array}
\right).
\end{equation}
The exchange degrees are $d_1 = 4$ and $d_i = 1$ for $i \in[2,6]$. The initial quiver associated with $\tilde{B}$ and exchange degrees are as follows:
\\
{\small
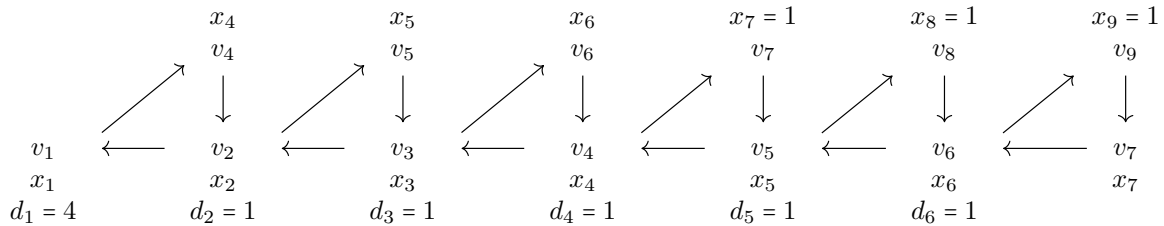
\begin{tikzcd}
	& \begin{array}{c} x_4\\v_4 \end{array} & \begin{array}{c} x_5\\v_5 \end{array} & \begin{array}{c} x_6\\v_6 \end{array} & \begin{array}{c} x_7=1\\v_7 \end{array} & \begin{array}{c} x_8=1\\v_8 \end{array} & \begin{array}{c} x_9=1\\v_9 \end{array} \\
	\begin{array}{c} v_1\\x_1\\d_1=4 \end{array} & \begin{array}{c} v_2\\x_2\\d_2=1 \end{array} & \begin{array}{c} v_3\\x_3\\d_3=1 \end{array} & \begin{array}{c} v_4\\x_4\\d_4=1 \end{array} & \begin{array}{c} v_5\\x_5\\d_5=1 \end{array} & \begin{array}{c} v_6\\x_6\\d_6=1 \end{array} & \begin{array}{c} v_7\\x_7\\\quad \end{array}
	\arrow[from=1-2, to=2-2]
	\arrow[from=1-3, to=2-3]
	\arrow[from=1-4, to=2-4]
	\arrow[from=1-5, to=2-5]
	\arrow[from=1-6, to=2-6]
	\arrow[from=1-7, to=2-7]
	\arrow[from=2-1, to=1-2]
	\arrow[from=2-2, to=1-3]
	\arrow[shift right=5, from=2-2, to=2-1]
	\arrow[from=2-3, to=1-4]
	\arrow[shift right=5, from=2-3, to=2-2]
	\arrow[from=2-4, to=1-5]
	\arrow[shift right=5, from=2-4, to=2-3]
	\arrow[from=2-5, to=1-6]
	\arrow[shift right=5, from=2-5, to=2-4]
	\arrow[from=2-6, to=1-7]
	\arrow[shift right=5, from=2-6, to=2-5]
	\arrow[shift right=5, from=2-7, to=2-6]
\end{tikzcd}}
\end{example}

For $\mathcal{C}_{\varepsilon, \xi}$, in the $k=2$, $\ell\in \mathbb{Z}_{\geq 2}$ case, and in the $k=3$ and $\ell=2$ case, Gleitz showed that the Grothendieck ring $K_0(\mathcal{C}_{\varepsilon, \xi})$ is a generalized cluster algebra structure of finite Dynkin types $C_{\ell-1}$ and $G_2$, respectively \cite[Theorem 4.1 and 5.3]{Gle16}. Moreover, for $k=3$, Gleitz made the following conjecture.

\begin{conjecture}[{\cite[Conjecture 5.4]{Gle16}}]
For $l\geq 2$, the Grothendieck ring of $\mathcal{C}_{\varepsilon, \xi}$ is isomorphic to a generalized cluster algebra of rank $2\ell-2$. Moreover, the generalized cluster monomials are mapped to classes of simple modules.
\end{conjecture}

This is a special case of Fraser's conjecture, see also \cite [Section 8.2]{Fra20}.

\section{Main results}\label{Section: main result}
In this section, we present the main results of this paper, including a classification of real Kirillov--Reshetikhin modules of $U_\varepsilon^{\res}(L\mathfrak{sl}_3)$, mutation sequences of real Kirillov--Reshetikhin modules of $U_\varepsilon^\res({L\mathfrak{sl}_{3}})$, and a proof of the first part of Gleitz's conjecture.

\subsection{\texorpdfstring{A classification of real Kirillov--Reshetikhin modules of $U_\varepsilon^\res({L\mathfrak{sl}_{3}})$}{a classification of real Kirillov--Reshetikhin module}}

\begin{lemma}\label{Lem: cofficients of $A_2$ are not real}
Let $\varepsilon^{2 \ell}=1$ with $\ell\geq 2$. The $U_\varepsilon^\res({L\mathfrak{sl}_{3}})$-module $L((X^{(\ell)}_{i, k})^c)$ is not real, where $i\in\{1,2\}$, $(i,k)\in \mathcal{X}$, and $c\in\ZZ_{\geq 1}$.
\end{lemma}

\begin{proof}
By the definition of $\mathcal{X}$ in Formula (\ref{Eq:X}), $(i,k)\in \mathcal{X}$ implies that $i-k \equiv 1 \pmod{2}$. By the definitions of $X^{(s)}_{i,k}$ in Equation (\ref{Eq:KR module}) and ${\bf Y}_{i, a}$ in Equation (\ref{mathbf{Y}}), $X^{(\ell)}_{i, k} = Y_{i,k} \cdots Y_{i,k+2\ell-2} = Y_{i, a\varepsilon^k} Y_{i, a\varepsilon^{k+2}} \cdots Y_{i, a\varepsilon^{k+2\ell-2}} = {\bf Y}_{i, a}$, where $a \in \mathbb{C}^{\times}$ is fixed in the beginning. Since $\varepsilon^{2\ell}=1$, $X^{(\ell)}_{i, k} = X^{(\ell)}_{i, k'}$ for any $(i,k)$, $(i,k') \in \mathcal{X}$. We prove the lemma in the case $i=1$; the remaining cases are analogous. By the definition of a real module, it suffices to show that
\begin{align}\label{inequalty}
\chi_{\varepsilon}(L((X^{(\ell)}_{1, k})^{2c}))\neq  \chi_{\varepsilon}(L((X^{(\ell)}_{1, k})^c))^2.
\end{align}
Recall that for a simple Lie algebra $\mathfrak{g}$ and a dominant weight $\mu$, we denote by $V_{\mu}$ the irreducible highest weight $\mathfrak{g}$-module of highest weight $\mu$, and denote by $\chi(V_{\mu})$ its ordinary character. According to Theorem \ref{Th:decomposition} and the discussion below it, the $\varepsilon$-character of $L((X^{(\ell)}_{1, k})^{b})$ for $b \in \ZZ_{\ge 1}$ can be computed as follows: compute the character of $V_{b\omega_1}$ using the Weyl character formula, and replace $y_i$ by ${\bf Y}_{i,a}$.
It follows that
\[
\dim(V_{a\omega_1})=\dim(L((X^{(\ell)}_{1, k})^c))),\,\, \dim(V_{2c\omega_1})=\dim(L((X^{(\ell)}_{1, k})^{2c})).
\]
By Weyl's dimension formula (see \cite[Section VI]{Hu1978}), for any $b\in\ZZ_{\geq 1}$, we have that $\dim(V_{b\omega_1})=\frac{(b+1)(b+2)}{2}$. Thus,
\begin{align*}
\dim(V_{c\omega_1})^2&=\big(\frac{(c+1)(c+2)}{2}\big)^2=\frac{(c+1)^2(c+2)^2}{4},\\
\dim(V_{2c\omega_1})&=\frac{(2c+1)(2c+2)}{2}=(2c+1)(c+1).
\end{align*}
Since $\dim(V_{c\omega_1})^2\neq \dim(V_{2c\omega_1})$, it follows that $\dim(L((X^{(\ell)}_{1, k})^c))^2\neq \dim(L((X^{(\ell)}_{1, k})^{2c}))$. Thus, we conclude that the inequality (\ref{inequalty}) is true, verifying our assertion.
\end{proof}

Furthermore, we have the following conclusion.
\begin{theorem}\label{Th: real KR module}
Let $\varepsilon^{2 \ell}=1$ with $\ell\geq 2$, $i \in \{1,2\}$, $(i,k)\in \mathcal{X}$, and $t \in \ZZ_{\ge 1}$. Then the Kirillov--Reshetikhin module $L(X^{(t)}_{i,k})$ of $U_\varepsilon^\res({L\mathfrak{sl}_{3}})$ is real if and only if $t\in[1,\ell-1]$. 
\end{theorem}

\begin{proof}
We will first demonstrate that the Kirillov--Reshetikhin module $L(X^{(t)}_{i,k})$ of $U_\varepsilon^\res({L\mathfrak{sl}_{3}})$ is real, where $i\in\{1,2\}$, $t\in[1,\ell-1]$, and $k\in \ZZ_{\geq 0}$. We will focus on the case $i=1$ and $k=0$; the proofs for the other cases will follow a similar approach.

According to the definition of a real module, it suffices to show that for $t\in[1,\ell-1]$,
\begin{align}\label{Eq:KR of $A_2$} 
\chi_\varepsilon(L(X^{(t)}_{1,0}))^2=\chi_\varepsilon(L((X^{(t)}_{1,0})^2)).
\end{align}
For the module $L((X^{(t)}_{1,0})^2)$, we apply the result from \cite[Theorem 1.1]{ALLZ25} to express it as the snake module $L(Y_{1,0}Y_{1,2}\cdots Y_{1,2t-2}Y_{1,2 \ell}Y_{1,2 \ell+2}\cdots Y_{1,2\ell+2t-2})$. On the one hand, the lowest path of $\mathscr{P}_{1,2t-2}$ and the highest path of $\mathscr{P}_{1,2\ell}$ are non-overlapping. On the other hand, for any monomial $m_1$ in $\chi_\varepsilon(L(Y_{1,0}Y_{1,2}\cdots Y_{1,2t-2}))$, there exists no monomial $m_2$ in $\chi_\varepsilon(L(Y_{1,2 \ell}Y_{1,2 \ell+2}\cdots Y_{1,2 \ell+2t-2}))$ such that $m_1m_2$ is a dominant monomial except the highest weight monomial. Therefore,  
\[
\chi_\varepsilon(L(Y_{1,0}\cdots Y_{1,2t-2}Y_{1,2 \ell}\cdots Y_{1,2\ell+2t-2}))= \chi_\varepsilon(L(Y_{1,0}\cdots Y_{1,2t-2}))\, \chi_\varepsilon(L(Y_{1,2 \ell}\cdots Y_{1,2\ell+2t-2})).
\]
Thus, Equation (\ref{Eq:KR of $A_2$}) holds.

For $t=\ell$, Lemma~\ref{Lem: cofficients of $A_2$ are not real} implies that the Kirillov--Reshetikhin module $L(X^{(\ell)}_{i,k})$ is not real. For $t>\ell$, we can write $t=c \ell+d$ for some $c\in \mathbb{Z}_{\geq 1}$, $d\in[0, \ell)$. By Theorem \ref{Th:decomposition}, we have
\[
\chi_{\varepsilon}(L(X^{(t)}_{i, k}))=\chi_{\varepsilon}(L((X^{(\ell)}_{i, k})^c)) \chi_{\varepsilon}(L(X^{(d)}_{i, k})).
\]
According to Lemma \ref{Lem: cofficients of $A_2$ are not real}, the module $L((X^{(\ell)}_{i, k})^c)$ is not real. Therefore $ \chi_\varepsilon(L((X^{(\ell)}_{i, k})^c))^2>\chi_\varepsilon(L((X^{(\ell)}_{i, k})^{2c})))$. Throughout, the notation
$\chi_\varepsilon(L(M))>\chi_\varepsilon(L(M'))$
means that the multiset of monomials appearing in
$\chi_\varepsilon(L(M'))$ is a proper submultiset of that appearing in
$\chi_\varepsilon(L(M))$, , while $\chi_\varepsilon(L(M))\ge \chi_\varepsilon(L(M'))$ means that every monomial of $\chi_\varepsilon(L(M'))$ appears in $\chi_\varepsilon(L(M))$ with at least the same multiplicity. It follows that
\begin{align*}
\chi_{\varepsilon}(L(X^{(t)}_{i, k}))^2 &= \chi_{\varepsilon}(L((X^{(\ell)}_{i, k})^c))^2\chi_{\varepsilon}(L(X^{(d)}_{i, k}))^2\\
&\ge \chi_{\varepsilon}(L((X^{(\ell)}_{i, k})^c))^2 \chi_{\varepsilon}(L((X^{(d)}_{i, k})^2))\\ 
&>\chi_{\varepsilon}(L((X^{(\ell)}_{i, k})^{2c})) \chi_{\varepsilon}(L((X^{(d)}_{i, k})^2))\\
&\ge \chi_{\varepsilon}(L((X^{(t)}_{i, k})^2))).
\end{align*}
Therefore, we conclude that the Kirillov--Reshetikhin module $L(X^{(t)}_{i, k})$ is not real for all $t \ge \ell$.
\end{proof}

\subsection{\texorpdfstring{Mutation equivalence of two seeds}{Mutation equivalence of cluster variables}}

To better describe mutation sequences and to visualize them more intuitively, and to make the exchange relations look more symmetric and neater, we introduce another initial seed of the generalized cluster algebra in Fraser’s conjecture (Conjecture \ref{Conj: rewrite Fraser's conjecture}). We will prove that this initial seed is mutation equivalent to the initial seed in Fraser’s conjecture (Conjecture \ref{Conj: rewrite Fraser's conjecture}).

\begin{definition}\label{Def:generalized cluster algebras of $A_2$}
Let $\mathcal{A}$ be a generalized cluster algebra of rank $2 \ell-2$ with initial seed $(\mathbf{x},\rho, B)$, defined as follows.
The initial cluster is 
\[
\mathbf{x} = \{x_1, x_2, \ldots, x_{2\ell-2}\},
\]
where the cluster variables are given by
\begin{align*}
x_1 =& [L(Y_{1,4}Y_{1,6}\cdots Y_{1, 2 \ell}Y_{2,2\ell+3}Y_{2,2\ell+5}\cdots Y_{2,4\ell -1})],\\
x_2 =& [L(Y_{2,2\ell+3}Y_{2,2\ell+5}\cdots Y_{2,4\ell -1})], \quad
x_3 = [L(Y_{2,2\ell+3}Y_{2,2\ell+5}\cdots Y_{2,4\ell -3})], \ \ldots, \ 
x_\ell = [L(Y_{2,2\ell+3})],\\
x_{\ell+1}& = [L(Y_{1,4}Y_{1,6}\cdots Y_{1, 2 \ell-2})], \qquad
x_{\ell+2} = [L(Y_{1,4}Y_{1,6}\cdots Y_{1, 2 \ell-4})], \,\quad\ldots, \,\quad
x_{2\ell-2} = [L(Y_{1,4})].
\end{align*}
The exchange degrees are $d_1=3$, $d_i=1$ for $i\in[2,2\ell-2]$. In addition, $\rho=\{\rho_1,\,\rho_2,\,\ldots,\,\rho_{2\ell-2}\}$, $\rho_1=(1,\,\rho_{1,1},\, \rho_{1,2},\, 1)$, $\rho_{1,1}=[L(Y_{1,0}Y_{1,2}\cdots Y_{1,2\ell-2})]$, $\rho_{1,2}=[L(Y_{2,1}Y_{2,3}\cdots Y_{2,2\ell-1})]$, and $\rho_i=(1,1)$ for $i\in[2, 2\ell-2]$. The exchange matrix $B$ is
\begin{equation}\label{exchhange matrix}
 B=
\bordermatrix{
&1  &2   &3   &4 &\dots &\ell & \ell+1 &&\dots &&& 2\ell-2\cr\\
1 &0&1&0&0&\dots&0&-1&0&0 &0 & \dots &0\cr\\
2 &-1&0&-1&0&\dots &0 &2&0& 0&0 & \dots &0\cr \\
3 &0&1&0&-1&\ddots& 0&-1&1&0  &0 & \dots &0 \cr     \\
4 &0&0&1&0&\ddots&0&0&-1&1&0 & \dots &0\cr            \\
\vdots&\vdots&\ddots&\ddots&\ddots&\ddots&\ddots&\ddots&\ddots&\ddots&\ddots&\ddots&\vdots\cr\\
&0&\dots&0&1&0&-1&0&0&\dots&0&-1&1\cr\\
\ell &0&\dots&0&0&1&0&0&0&\dots&0&0&-1\cr\\
\ell +1 &1&-2&1&0&\ddots&0&0&-1&0&\dots& \dots &0\cr\\
&0&0&-1&1 &\ddots&0&1&0&-1&0&\dots  &0\cr\\
\vdots &\vdots&\ddots&\ddots&\ddots&\ddots&\ddots&\ddots&\ddots& \ddots&\ddots&\ddots&\ddots\cr\\
&0&\dots &0&-1 &1&0&0&0&\dots&    1&0&-1\cr\\
2\ell -2 &0&\dots &0&0 &-1&1&0&0&\dots&    0&1&0\cr\\}.
\end{equation}
The initial quiver $Q$ associated with $B$ and exchange degrees is shown in Figure \ref{fig:initial quiver of A}. 
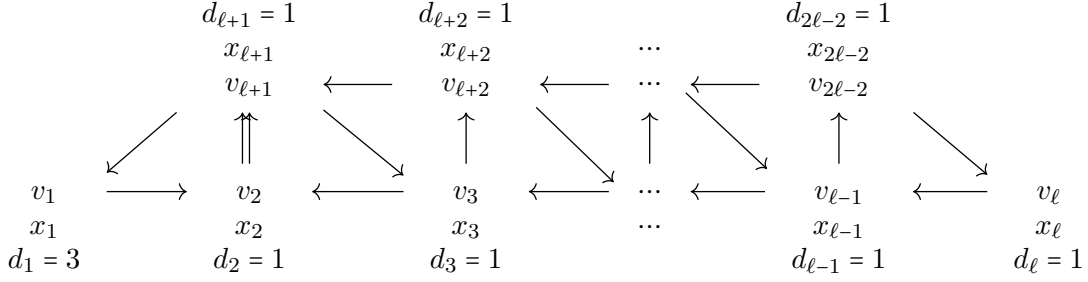
\begin{figure}[H]
\begin{center}
\begin{tikzcd}
	& \begin{array}{c} d_{\ell+1}=1\\x_{\ell+1}\\v_{\ell+1} \end{array} & \begin{array}{c} d_{\ell+2}=1\\x_{\ell+2}\\v_{\ell+2} \end{array} & \begin{array}{c} \quad\\\cdots\\\cdots \end{array} & \begin{array}{c} d_{2\ell-2}=1\\x_{2\ell-2}\\v_{2\ell-2} \end{array} & \\
	\begin{array}{c} v_1\\x_1\\d_1=3 \end{array} & \begin{array}{c} v_2\\x_2\\d_2=1 \end{array} & \begin{array}{c} v_3\\x_3\\d_3=1 \end{array} & \begin{array}{c} \cdots\\\cdots\\\quad \end{array} & \begin{array}{c} v_{\ell-1}\\x_{\ell-1}\\d_{\ell-1}=1 \end{array} & \begin{array}{c} v_\ell\\x_\ell\\d_\ell=1 \end{array}
	\arrow[from=1-2, to=2-1]
	\arrow[from=1-2, to=2-3]
	\arrow[shift left=5, from=1-3, to=1-2]
	\arrow[shift left, from=1-3, to=2-4]
	\arrow[shift left=5, from=1-4, to=1-3]
	\arrow[shift right, from=1-4, to=2-5]
	\arrow[shift left=5, from=1-5, to=1-4]
	\arrow[from=1-5, to=2-6]
	\arrow[shift left=5, from=2-1, to=2-2]
	\arrow[from=2-2, to=1-2]
	\arrow[shift left, from=2-2, to=1-2]
	\arrow[from=2-3, to=1-3]
	\arrow[shift right=5, from=2-3, to=2-2]
	\arrow[from=2-4, to=1-4]
	\arrow[shift right=5, from=2-4, to=2-3]
	\arrow[from=2-5, to=1-5]
	\arrow[shift right=5, from=2-5, to=2-4]
	\arrow[shift right=5, from=2-6, to=2-5]
\end{tikzcd}
\end{center}
\caption{ The initial quiver $Q$ and exchange degrees of $\mathcal{A}$.} \label{fig:initial quiver of A}
\end{figure}
\end{definition}

\begin{lemma}\label{Lem:mutation-equivalent}
The initial seeds in Definition~\ref{Def:generalized cluster algebras of $A_2$} and in Conjecture~\ref{Conj: rewrite Fraser's conjecture} (for $k=3$) are mutation-equivalent. 
\end{lemma}

\begin{proof}
Denote by $\mathcal{A}'$ the generalized cluster algebra in Conjecture \ref{Conj: rewrite Fraser's conjecture} for $k=3$. The initial seed ($\mathbf{x}', Q'$) of $\mathcal{A}'$ is given as follows. The initial cluster is $\mathbf{x}'= \{x'_1, x'_2, \ldots, x'_{2\ell-2}\}$, where 
{\small
\begin{align*}
&x'_1=[L(Y_{1,4}Y_{1,6} \cdots Y_{1,2 \ell}Y_{2,2 \ell+3} \cdots Y_{2,4 \ell-1})], \quad\, x'_2=[L(Y_{1,4}Y_{1,6} \cdots Y_{1,2 \ell}Y_{2,2 \ell+3} \cdots Y_{2,4 \ell-3})],\quad \ldots,\\
&x'_{\ell-1}=[L(Y_{1,4}Y_{1,6}\cdots Y_{1,2 \ell}Y_{2,2 \ell+3})], \quad x'_{\ell}=[L(Y_{1,4}Y_{1,6} \cdots Y_{1,2 \ell})], \quad x'_{\ell+1}=[L(Y_{1,4}Y_{1,6} \cdots Y_{1,2 \ell-2})],\quad \ldots, \\
&x'_{2\ell-3}=[L(Y_{1,4}Y_{1,6})], \quad  x'_{2\ell-2}=[L(Y_{1,4})].
\end{align*}}
Since the frozen variables $x_{(k-1)(\ell-1)+1},\dots,x_{(k-1)(\ell-1)+\ell}$ in $\mathcal{A}'$ are identified with $1$, we can ignore these frozen variables. The initial quiver $Q'$ and exchange degrees of $\mathcal{A}'$ are shown in Figure \ref{fig: initial quiver of A'}.

\begin{figure}[H]
\begin{center}
{\small
\begin{tikzcd}
	& \begin{array}{c} d_{\ell+1}=1\\x'_{\ell+1}\\v'_{\ell+1} \end{array} & \begin{array}{c} d_{\ell+2}=1\\x'_{\ell+2}\\v'_{\ell+2} \end{array} & \begin{array}{c} \quad\\\cdots\\\cdots \end{array} & \begin{array}{c} d_{2\ell-2}=1\\x'_{2\ell-2}\\v'_{2\ell-2} \end{array} && \\
	\begin{array}{c} v'_1\\x'_1\\d_1=3 \end{array} & \begin{array}{c} v'_2\\x'_2\\d_2=1 \end{array} & \begin{array}{c} v'_3\\x'_3\\d_3=1 \end{array} & \begin{array}{c} \cdots\\\cdots\\\quad \end{array} & \begin{array}{c} v'_{\ell-1}\\x'_{\ell-1}\\d_{\ell-1}=1 \end{array} & \begin{array}{c} v'_\ell\\x'_\ell\\d_\ell=1 \end{array} & \begin{array}{c} v'_{\ell+1}\\x'_{\ell+1}\\\quad \end{array}
	\arrow[from=1-2, to=2-2]
	\arrow[shift left=5, from=1-3, to=1-2]
	\arrow[from=1-3, to=2-3]
	\arrow[shift left=5, from=1-4, to=1-3]
	\arrow[from=1-4, to=2-4]
	\arrow[shift left=5, from=1-5, to=1-4]
	\arrow[from=1-5, to=2-5]
	\arrow[from=2-1, to=1-2]
	\arrow[from=2-2, to=1-3]
	\arrow[shift right=5, from=2-2, to=2-1]
	\arrow[shift right, from=2-3, to=1-4]
	\arrow[shift right=5, from=2-3, to=2-2]
	\arrow[shift left, from=2-4, to=1-5]
	\arrow[shift right=5, from=2-4, to=2-3]
	\arrow[shift right=5, from=2-5, to=2-4]
	\arrow[shift right=5, from=2-6, to=2-5]
	\arrow[shift right=5, from=2-7, to=2-6]
\end{tikzcd}}
\end{center}
\caption{The initial quiver $Q'$ and exchange degrees of $\mathcal{A}'$. Note that in this figure, there are two vertices labeled by $v'_{\ell+1}$.} \label{fig: initial quiver of A'}
\end{figure}
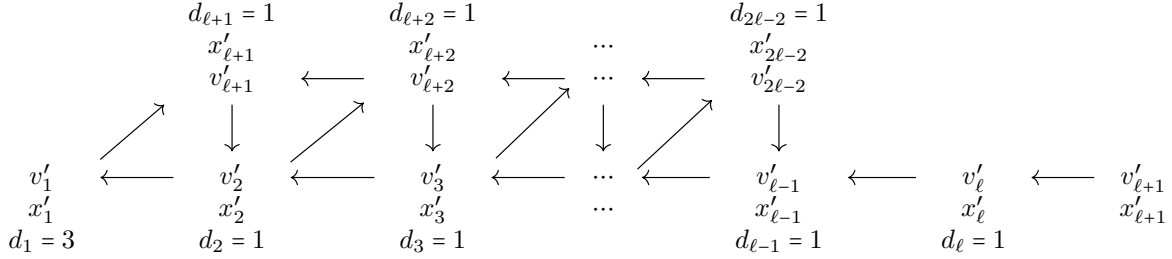

We first show that the cluster variables $x_1,x_2, \ldots, x_{2 \ell-2}$, introduced in Definition \ref{Def:generalized cluster algebras of $A_2$}, can be obtained from the cluster variables $x'_1,x'_2,\ldots, x'_{2\ell-2}$ via a sequence of mutations, where
\begin{align*}
x_1=&[L(Y_{1,4}Y_{1,6}\cdots Y_{1, 2 \ell}Y_{2,2\ell+3}Y_{2,2\ell+5}\cdots Y_{2,4\ell -1})], \hspace{19em} \\  
x_2=&[L(Y_{2,2\ell+3}Y_{2,2\ell+5}\cdots Y_{2,4\ell -1})], \quad
x_3=[L(Y_{2,2\ell+3}Y_{2,2\ell+5}\cdots Y_{2,4\ell -3})], \quad \ldots,\quad  
x_\ell=[L(Y_{2,2\ell+3})],\\
x_{\ell+1}&= [L(Y_{1,4}Y_{1,6}\cdots Y_{1, 2 \ell-2})],\quad\quad
x_{\ell+2}= [L(Y_{1,4}Y_{1,6}\cdots Y_{1, 2 \ell-4})], \quad\quad \ldots,  \quad\,\, 
x_{2\ell-2}=[L(Y_{1,4})].
\end{align*}

Since $x_i=x'_i$ for $i\in[\ell+1, 2 \ell-2]\cup \{1\}$, it suffices to mutate the variables $x'_2,x'_3, \ldots, x'_\ell$. We begin with the cluster variable $x'_\ell$. The exchange relation for the mutation of $x'_\ell$ is 
\[
x'_\ell \,y'_\ell=x'_{\ell-1}+x'_{\ell+1},
\]
that is, 
\[
\chi_\varepsilon(L(Y_{1,4}Y_{1,6} \cdots Y_{1,2 \ell}))\,\chi_\varepsilon( y'_\ell)=\chi_\varepsilon(L(Y_{1,4}Y_{1,6}\cdots Y_{1,2 \ell}Y_{2,2 \ell+3}))+\chi_\varepsilon(L(Y_{1,4}Y_{1,6} \cdots Y_{1,2 \ell-2})),
\]
where $y'_\ell=[L(Y_{2,2 \ell+3})]$. So, $y'_\ell=x_\ell$. Next, we mutate the cluster variable $x'_{\ell-1}$. The exchange relation is
\[
x'_{\ell-1}\, y'_{\ell-1}=x_{\ell}\,x'_{\ell-2}+x'_{\ell+1}\,x'_{2 \ell-2},
\]
that is,
\begin{align*}
\chi_\varepsilon(L(Y_{1,4}Y_{1,6}\cdots Y_{1,2 \ell}Y_{2,2 \ell+3}))\, \chi_\varepsilon(y'_{\ell-1})=& \chi_\varepsilon(L(Y_{2,2 \ell+3}))\,\chi_\varepsilon(L(Y_{1,4}Y_{1,6}\cdots Y_{1,2 \ell}Y_{2,2 \ell+3}Y_{2,2 \ell+5}))\\
& +\chi_\varepsilon(L(Y_{1,4}Y_{1,6} \cdots Y_{1,2 \ell-2}))\,\chi_\varepsilon(L(Y_{1,4})),
\end{align*}
where $y'_{\ell-1}=[L(Y_{2,2 \ell+3}Y_{2,2 \ell+5})]$. So, $y'_{\ell-1}=x_{\ell-1}$. Proceeding in the same way, by mutating cluster variables $x'_{\ell-2}$, $x'_{\ell-3}, \ldots, x'_2$ in succession, we obtain the cluster variables $x_{\ell-2}$, $x_{\ell-3}, \ldots, x_2$, respectively. 

Now we check that after the above sequence of mutations, the initial quiver of $\mathcal{A}'$ becomes the initial quiver of $\mathcal{A}$. 
Let $\mu_i(Q')$ denote the mutation at the $i$-th vertex of quiver $Q'$. According to quiver mutation introduced in Definition \ref{Def: quiver mutation}, we have that $\mu_2\mu_3\cdots\mu_{\ell-1}\mu_{\ell}(Q')$ is the same as $Q$, see Figure \ref{fig:initial quiver of A}.

Thus, the initial seeds in Definition \ref{Def:generalized cluster algebras of $A_2$} and in Conjecture \ref{Conj: rewrite Fraser's conjecture} (for $k=3$) are mutation-equivalent.
\end{proof}

\begin{remark}
Following Lemma \ref{Lem:mutation-equivalent}, we have that the generalized cluster algebras introduced in Conjecture \ref{Conj: rewrite Fraser's conjecture} (for $k=3$) and Definition \ref{Def:generalized cluster algebras of $A_2$} coincide. 
\end{remark}

\subsection{\texorpdfstring{Mutation sequences for real Kirillov--Reshetikhin $U_{\varepsilon}^{\res}({L\mathfrak{sl}_{3}})$-modules}{mutation sequence for real KR modules}}

For $(i,k)$, $(j,v)\in \mathcal{X}$, $s,t\in\ZZ_{\geq 1}$, we denote
\begin{align}\label{Eq:AFM}
M^{j,v,t}_{i,k,s}=X_{i,k}^{(s)}X_{j,v}^{(t)}=Y_{i,k}Y_{i,k+2}\cdots Y_{i,k+2(s-1)}Y_{j,v}Y_{j,v+2}\cdots Y_{j,v+2(t-1)},
\end{align}
where $X_{i,k}^{(s)}$ is defined in Equation (\ref{Eq:KR module}). By the definition of $M^{j,v,t}_{i,k,s}$, we have $M^{j,v,t}_{i,k,s}=M^{i,k,s}_{j,v,t}$. According to Proposition \ref{Prop:minimal affine module}, when $v=k+2s+|j-i|$, the module $L(M^{j,v,t}_{i,k,s})$ is a minimal affinization. We denote 
\begin{align}\label{R}
R_{(i,k,s)(j,v,t)}=\sum_{(p_{1},\ldots,p_{s+t}) \in \overline{\mathscr{P}}_{((i,k+2r)_{0 \leq r \leq s-1}(j,v+2r)_{0 \leq r \leq t-1})}} \prod_{z=1}^{s+t}\mathfrak{m}(p_{z}),
\end{align}
where $\overline{\mathscr{P}}_{(i_t,k_t)_{1\le t\le s+t}}$ is defined in \eqref{Non-overlapping paths},
with $(i_t,k_t)$ given by the concatenation $(i,k+2r)_{0 \le r \le s-1}$ and $(j,v+2r)_{0 \le r \le t-1}$.

\begin{theorem}\label{Th:KR mutation sequence}
Let $\mathcal{A}$ be the generalized cluster algebra defined in Definition \ref{Def:generalized cluster algebras of $A_2$}. Consider the two mutation sequences
\[
(v_1, v_2, v_{\ell+1}, v_{\ell+2}, \ldots, v_{2\ell-2}) \quad \text{and} \quad (v_1, v_2, v_3, v_4, \ldots, v_\ell).
\]
By successively applying these two mutation sequences in an alternating block-wise manner, one obtains a family of cluster variables in $\mathcal{A}$. Among the cluster variables obtained in this way, all cluster variables corresponding to real Kirillov--Reshetikhin modules of $U_\varepsilon^\res(L\mathfrak{sl}_3)$ appear. Moreover, the resulting infinite mutation process is periodic and returns to the initial cluster.
\end{theorem}

The proof of Theorem \ref{Th:KR mutation sequence} will be given at the end of this subsection. In the following, we present several relations satisfied by modules of $U_{\varepsilon}^{\res}({L\mathfrak{sl}_{3}})$.

\begin{lemma}\label{Lem: minimal affine is special}
For $U_{\varepsilon}^{\res}({L\mathfrak{sl}_{3}})$, let $\varepsilon^{2 \ell}=1$ with $\ell\in \ZZ_{\geq 2}$. Let $\{i,j\}=\{1,2\}$ and let $(i,k), (j,v)\in\mathcal{X}$ with $v=k+2 \ell-1$. Then, for $1\leq t \leq \ell-1$, 
\begin{align*}
\chi_\varepsilon(L(M^{j,v,t}_{i,k,\ell-1}))=\chi_q(L(M^{j,v,t}_{i,k,\ell-1}))|_{q=\varepsilon}.
\end{align*}
Moreover, for $t=\ell-1$, the minimal affinizations $L(M^{j,v,t}_{i,k,\ell-1})$ are special.
\end{lemma}

\begin{proof}
We prove the lemma in the case $i=1$, $j=2$, and $k=0$; the proofs for the other cases follow a similar approach. For $t\leq \ell-1$, we need to prove the identity 
\[
\chi_\varepsilon(L(M^{2, 2\ell-1,t}_{1,0,\ell-1}))=\chi_q(L(M^{2,2\ell-1,t}_{1,0,\ell-1}))|_{q=\varepsilon},
\]
that is,
{\small
\begin{align*}
\chi_\varepsilon(L(Y_{1,0}Y_{1,2}\cdots Y_{1,2\ell-4}Y_{2,2\ell-1}Y_{2,2\ell+1}\cdots Y_{2,2 \ell+2t-3}))=\chi_q(L(Y_{1,0}Y_{1,2}\cdots Y_{1,2\ell-4}Y_{2,2\ell-1}Y_{2,2\ell+1}\cdots Y_{2,2\ell+2t-3}))|_{q=\varepsilon}.
\end{align*}}
To establish the above identity, it suffices to show that $R_{(1,0,\ell-1)(2,2\ell-1,t)}$ contains no dominant monomial $m$ except the highest weight monomial such that the $\varepsilon$-character of $L(m)$ is contained in $R_{(1,0,\ell-1)(2,2\ell-1,t)}$. Monomials in $R_{(1,0,\ell-1)(2,2\ell-1,t)}$ are of the form $\mathfrak{m}(p_1)\cdots \mathfrak{m}(p_{t+\ell-1})$, where  $p_i \in \mathscr{P}_{1,2i-2}$ for $i\in[1, \ell)$, and $p_i \in \mathscr{P}_{2,2i-1}$ for $i\in[\ell, t+\ell-1]$, see Formula (\ref{R}). The proof is divided into the following three cases.

{\bf Case} 1. Suppose that $p_a\in \mathscr{P}_{1,2a-2}$ with $\mathfrak{m}(p_a)=Y_{1,2a-2}$ for some $a\in[1, \ell)$. To obtain a dominant monomial, a path $p_b\in \mathscr{P}_{2,2b-1}$, for some $b\in[\ell,\ell+t-1]$, must either consist entirely of upper corners, or have $(1,2a-2+2 \ell)$ as its unique lower corner. If $p_b$ consists entirely of upper corners, then $\mathfrak{m}(p_b)=Y_{2,2b-1}$, where $b\in[\ell,\ell+t-1]$. In this case, we obtain the dominant monomial $Y_{1,0}Y_{1,2}\cdots Y_{1,2\ell-4}Y_{2,2\ell-1}Y_{2,2\ell+1}\cdots Y_{2,2\ell+2t-3}$, which is the highest weight monomial. If $p_b$ (with $b\in[\ell,\ell+t-1]$) has $(1,2a-2+2\ell)$ as its unique lower corner, then $a\geq 2$ and $\mathfrak{m}(p_b)=Y^{-1}_{1,2a-2+2 \ell}=Y^{-1}_{1,2b+2}$, where $p_b\in\mathscr{P}_{2,2b-1}$. By Formula (\ref{Non-overlapping paths}), we have $\mathfrak{m}(p_u)=Y^{-1}_{1,2u+2}$, where $p_u\in\mathscr{P}_{2,2u-1}$, $u\in [b, \ell+t-1]$.

In other words, in order to obtain a dominant monomial, for $a\in[1,\ell)$, we take $\mathfrak{m}(p_a)=Y_{1,2a-2}$. For $z\in[\ell,b)$, we take $\mathfrak{m}(p_z)=Y_{2,2z-1}$. For $u\in[b, \ell+t-1]$, we take $\mathfrak{m}(p_u)=Y^{-1}_{1, 2u+2}$. Thus, we obtain the monomial 
\begin{align}\label{monomial}
Y_{1,0}\cdots Y_{1,2 \ell-4}Y_{2,2\ell-1}\cdots Y_{2,2b-3} Y^{-1}_{1,2b+2}\cdots Y^{-1}_{1,2 \ell+2t}.
\end{align}

If $t=\ell-1$, then, since $a\geq 2$, the factor $\mathfrak{m}(p_{\ell+t-1})=Y^{-1}_{1,2 \ell+2t}=Y^{-1}_{1,4 \ell-2}$ cannot be canceled by any other factor appearing in the product. Therefore, the monomial (\ref{monomial}) is not dominant.

If $t<\ell-1$, then the monomial (\ref{monomial}) simplifies to 
\[
Y_{1,0}\cdots Y_{1,2a-4}Y_{1,2a+2t+2}\cdots Y_{1,2\ell-4}Y_{2,2 \ell-1}\cdots Y_{2,2b-3}, \quad 
\]
which is a dominant monomial. Its lowest weight monomial is  
\[
Y^{-1}_{2,3} Y^{-1}_{2,5}\cdots Y^{-1}_{2,2 \ell-1} Y^{-1}_{1,2 \ell+2}\cdots Y^{-1}_{1,2b} \cdot Y_{2,2b-1}\cdots Y_{2,2 \ell+2t-3},
\]
which does not appear in $R_{(1,0,\ell-1)(2,2\ell-1,t)}$, since $Y^{-1}_{1,2b}$ and $Y_{2,2b-1}$ are overlapping. Thus, there is no dominant monomial $m$ except the highest weight monomial such that the $\varepsilon$-character of $L(m)$ is contained in $R_{(1,0,\ell-1)(2,2\ell-1,t)}$.

{\bf Case} 2. Suppose that $p_a\in \mathscr{P}_{1,2a-2}$ with $\mathfrak{m}(p_a)=Y^{-1}_{1,2a}Y_{2,2a-1}$ for $a\in[1, \ell)$. To obtain a dominant monomial, a path $p_b\in \mathscr{P}_{2,2b-1}$, for $b\in[\ell,t+\ell-1]$, must have $(1,2a+2+2 \ell)$ as its upper corner. In this case, we have 
\[
\mathfrak{m}(p_b)=Y_{1,2a+2 \ell}Y^{-1}_{2,2a+2 \ell+1}=Y_{1,2b}Y^{-1}_{2, 2b+1}.
\]
Hence
\[
\mathfrak{m}(p_a)\mathfrak{m}(p_b)=Y^{-1}_{1,2a}Y_{2,2a-1}Y_{1,2a+2 \ell}Y^{-1}_{2,2a+2 \ell+1}=Y_{2,2a-1}Y^{-1}_{2,2a+2 \ell+1}.
\]
To obtain a dominant monomial, let $\mathfrak{m}(p_{a+1})=Y^{-1}_{1,2a+2}Y_{2,2a+1}$. Therefore,
\[
\mathfrak{m}(p_a)\mathfrak{m}(p_b)\mathfrak{m}(p_{a+1})=Y_{2,2a-1}Y^{-1}_{2,2a+2 \ell+1}Y^{-1}_{1,2a+2}Y_{2,2a+1}=Y_{2,2a-1}Y^{-1}_{1,2a+2}.
\]
By Formula (\ref{Non-overlapping paths}), proceeding in this way, since $t\leq \ell-1$, we will always find that there exists a monomial, namely $Y^{-1}_{1,2 \ell-2}$ or $Y^{-1}_{2,2\ell+2t-1}$, which cannot be canceled. Thus, in this case, no dominant monomial appears in $R_{(1,0,\ell-1)(2,2\ell-1,t)}$.

{\bf Case} 3. Suppose that $p_a\in \mathscr{P}_{1,2a-2}$ with $\mathfrak{m}(p)=Y^{-1}_{2,2a+1}$ for some $a\in[1, \ell)$. To obtain a dominant monomial, the path $p_b\in \mathscr{P}_{2,2b-1}$, for some $b\in[\ell,\ell+t-1]$, must have $(2,2a+2\ell+1)$ as its upper corner. Then 
\[
\mathfrak{m}(p_b)=Y_{2,2a+2\ell+1}=Y_{2,2b-1}.
\]
Using the same method as in Case 1, we conclude that no dominant monomial appears in $R_{(1,0,\ell-1)(2,2\ell-1,t)}$ in this case.

In summary, we have
\begin{align*}
\chi_\varepsilon(L(M^{j,v,t}_{i,k,\ell-1}))=\chi_q(L(M^{j,v,t}_{i,k,\ell-1}))|_{q=\varepsilon}.
\end{align*}
\end{proof}  

\begin{lemma}\label{Lem:exchange relations}
For $U_{\varepsilon}^{\res}(L\mathfrak{sl}_3)$, let $\varepsilon^{2\ell} = 1$ with $\ell \in \mathbb{Z}_{\ge 3}$. Then, for $\{i,j\} = \{1,2\}$ and $(i,k) \in \mathcal{X}$, the following identities hold:
{\small
\begin{enumerate}
\item $\chi_\varepsilon(L(X^{(u+1)}_{i,k}))\,\chi_\varepsilon(L(X^{(u+1)}_{i,k+2}))=
\chi_\varepsilon(L(X^{(u+2)}_{i,k}))\,\chi_\varepsilon(L(X^{(u)}_{i,k+2}))+\chi_\varepsilon(L(X^{(u+1)}_{j,k+1}))$, \textit{ where } $0\leq u \leq \ell-3$;

\item $\chi_\varepsilon(L(X^{(\ell-1)}_{i,k}))\,\chi_\varepsilon(L(X^{(\ell-1)}_{j,k-1}))=\chi_\varepsilon(L(M^{j,k-1,\ell-1}_{i,k,\ell-1}))
+\chi_\varepsilon(L(X^{(\ell-2)}_{i,k}))\,\chi_\varepsilon(L(X^{(\ell-2)}_{j,k+1}))$; 

\item $\chi_\varepsilon(L(M^{j,k+1,\ell-1}_{i,k,\ell-1}))\,\chi_\varepsilon(L(M^{j,k-1,\ell-1}_{i,k,\ell-1}))
=\chi_\varepsilon(L(X^{(\ell-1)}_{i,k}))^3+\chi_\varepsilon(L(X^{(\ell-2)}_{j,k+1}))^3\\ +\chi_\varepsilon(L(X^{(\ell)}_{j,k-1}))\, \chi_\varepsilon(L(X^{(\ell-1)}_{i,k}))^2\,\chi_\varepsilon(L(X^{(\ell-2)}_{j,k+1}))
+\chi_\varepsilon(L(X^{(\ell)}_{i,k}))\,\chi_\varepsilon(L(X^{(\ell-1)}_{i,k}))\,\chi_\varepsilon(L(X^{(\ell-2)}_{j,k+1}))^2$.
\end{enumerate}}
\end{lemma}

\begin{proof}
(1) For $U_{\varepsilon}^{\res}({L\mathfrak{sl}_{3}})$ with $0\leq u \leq \ell-3$, according to \cite[Lemma 5.6]{ALLZ25}, we deduce that
\begin{align*}
& \chi_\varepsilon(L(X^{(u+1)}_{i,k}))=\chi_q(L(X^{(u+1)}_{i,k}))|_{q=\varepsilon},\,\,
 \chi_\varepsilon(L(X^{(u+1)}_{i,k+2}))=\chi_q(L(X^{(u+1)}_{i,k+2}))|_{q=\varepsilon},\\
& \chi_\varepsilon(L(X^{(u+2)}_{i,k}))=\chi_q(L(X^{(u+2)}_{i,k}))|_{q=\varepsilon},\,\,
 \chi_\varepsilon(L(X^{(u)}_{i,k+2}))=\chi_q(L(X^{(u)}_{i,k+2}))|_{q=\varepsilon},\\
& \chi_\varepsilon(L(X^{(u+1)}_{j,k+1}))=\chi_q(L(X^{(u+1)}_{j,k+1}))|_{q=\varepsilon}.
\end{align*}
According to the T-system of $U_{q}({L\mathfrak{sl}_{3}})$, we have the following identity of $q$-characters
\begin{align*}
&\chi_q(L(X^{(u+1)}_{i,k})) \chi_q(L(X^{(u+1)}_{i,k+2}))=\chi_q(L(X^{(u+2)}_{i,k})) \chi_q(L(X^{(u)}_{i,k+2}))+\chi_q(L(X^{(u+1)}_{j,k+1})).
\end{align*}
It follows that after the specialization $q=\varepsilon$, we have 
{\small
\begin{align*}
&\chi_q(L(X^{(u+1)}_{i,k}))|_{q=\varepsilon}\, \chi_q(L(X^{(u+1)}_{i,k+2}))|_{q=\varepsilon}=\chi_q(L(X^{(u+2)}_{i,k}))| _{q=\varepsilon}\,\chi_q(L(X^{(u)}_{i,k+2}))|_{q=\varepsilon}+\chi_q(L(X^{(u+1)}_{j,k+1}))|_{q=\varepsilon}.
\end{align*}}
Therefore, the conclusion is valid.

(2) We prove the case for $i=1$, $j=2$, and $k=2$. The proofs for the other cases are similar. The module $L(M^{2,1,\ell-1}_{1,2,\ell-1})$ can be converted into a minimal affinization $L(M^{2,2 \ell+1,\ell-1}_{1,2,\ell-1})$. 
To prove 
\begin{align*}
\chi_\varepsilon(L(X^{(\ell-1)}_{1,2}))\chi_\varepsilon(L(X^{(\ell-1)}_{2,1}))=\chi_\varepsilon(L(M^{2,1,\ell-1}_{1,2,\ell-1}))
+\chi_\varepsilon(L(X^{(\ell-2)}_{1,2}))\chi_\varepsilon(L(X^{(\ell-2)}_{2,3})),
\end{align*}
it suffices to prove
\[
\chi_\varepsilon(L(X^{(\ell-1)}_{1,2}))\chi_\varepsilon(L(X^{(\ell-1)}_{2,2\ell+1}))=\chi_\varepsilon(L(M^{2,2 \ell+1,\ell-1}_{1,2,\ell-1}))
+\chi_\varepsilon(L(X^{(\ell-2)}_{1,2}))\chi_\varepsilon(L(X^{(\ell-2)}_{2,2\ell+3})).
\]
By the proof of (1), for $t=\ell-2, \ell-1$, we have 
\[
\chi_\varepsilon(L(X^{(t)}_{1,2}))=\chi_q(L(X^{(t)}_{1,2}))|_{q=\varepsilon}, \quad
\chi_\varepsilon(L(X^{(t)}_{2,2 \ell+1}))=\chi_q(L(X^{(t)}_{2,2 \ell+1}))|_{q=\varepsilon}.
\]
By Lemma \ref{Lem: minimal affine is special}, we have shown that
\[
\chi_\varepsilon(L(M^{2,2 \ell+1,\ell-1}_{1,2,\ell-1}))=\chi_q(L(M^{2,2 \ell+1,\ell-1}_{1,2,\ell-1}))|_{q=\varepsilon}.
\] 
For $U_{q}({L\mathfrak{sl}_{3}})$, the following identity of $q$-characters
\[
\chi_q(L(X^{(\ell-1)}_{1,2}))\,\chi_q(L(X^{(\ell-1)}_{2,2\ell+1}))=\chi_q(L(M^{2,2\ell+1,\ell-1}_{1,2,\ell-1}))
+\chi_q(L(X^{(\ell-2)}_{1,2}))\,\chi_q(L(X^{(\ell-2)}_{2,2\ell+3}))
\]
holds.
It follows that after the specialization $q=\varepsilon$, we have 
{\small\[
\chi_q(L(X^{(\ell-1)}_{1,2}))|_{q=\varepsilon}\,\chi_q(L(X^{(\ell-1)}_{2,2\ell+1}))|_{q=\varepsilon}=\chi_q(L(M^{2,2 \ell+1,\ell-1}_{1,2,\ell-1}))|_{q=\varepsilon}
+\chi_q(L(X^{(\ell-2)}_{1,2}))|_{q=\varepsilon}\,\chi_q(L(X^{(\ell-2)}_{2,2 \ell+3}))|_{q=\varepsilon}.
\]}
Therefore, the conclusion is valid.

(3) We prove the case $i=1$, $j=2$, and $k=2$; the proofs for the other cases are analogous. In this case, it suffices to prove the following identity:
\begin{align}\label{Eq:first equivent}
\chi_\varepsilon(L(M^{2,3,\ell-1}_{1,2,\ell-1}))\chi_\varepsilon(L(M^{2,1,\ell-1}_{1,2,\ell-1})) 
&=\chi_\varepsilon(L(X^{(\ell-1)}_{1,2}))^3+\chi_\varepsilon(L(X^{(\ell)}_{2,1})) \chi_\varepsilon(L(X^{(\ell-1)}_{1,2}))^2\chi_\varepsilon(L(X^{(\ell-2)}_{2,3}))\nonumber\\
&+\chi_\varepsilon(L(X^{(\ell)}_{1,2}))\chi_\varepsilon(L(X^{(\ell-1)}_{1,2}))\chi_\varepsilon(L(X^{(\ell-2)}_{2,3}))^2+\chi_\varepsilon(L(X^{(\ell-2)}_{2,3}))^3.
\end{align}
Applying Equation (2), we obtain 
\begin{align}
\chi_\varepsilon(L(M^{2,1,\ell-1}_{1,2,\ell-1}))&=\chi_\varepsilon(L(X^{(\ell-1)}_{1,2}))\,\chi_\varepsilon(L(X^{(\ell-1)}_{2,1}))-\chi_\varepsilon(L(X^{(\ell-2)}_{1,2}))\,\chi_\varepsilon(L(X^{(\ell-2)}_{2,3})),\label{Eq:second identity 1}\\
\chi_\varepsilon(L(M^{2,3,\ell-1}_{1,2,\ell-1}))&=\chi_\varepsilon(L(M^{1,2,\ell-1}_{2,3,\ell-1}))\nonumber\\
&=\chi_\varepsilon(L(X^{(\ell-1)}_{2,3}))\chi_\varepsilon(L(X^{(\ell-1)}_{1,2}))-\chi_\varepsilon(L(X^{(\ell-2)}_{2,3}))\chi_\varepsilon(L(X^{(\ell-2)}_{1,4})).\quad\quad\quad\label{Eq:second identity 2}
\end{align}
Substituting Equations (\ref{Eq:second identity 1}) and (\ref{Eq:second identity 2}) into (\ref{Eq:first equivent}), we obtain
{\small
\begin{align}\label{Eq:third identity}
\chi_\varepsilon(L(M^{2,3,\ell-1}_{1,2,\ell-1}))\,\chi_\varepsilon(L(M^{2,1,\ell-1}_{1,2,\ell-1}))&=\chi_\varepsilon(L(X^{(\ell-1)}_{1,2}))^2\chi_\varepsilon(L(X^{(\ell-1)}_{2,1}))\chi_\varepsilon(L(X^{(\ell-1)}_{2,3}))\nonumber\\
&+\chi_\varepsilon(L(X^{(\ell-2)}_{1,2}))\chi_\varepsilon(L(X^{(\ell-2)}_{1,4}))\chi_\varepsilon(L(X^{(\ell-2)}_{2,3}))^2\nonumber\\
&-\chi_\varepsilon(L(X^{(\ell-1)}_{1,2}))\chi_\varepsilon(L(X^{(\ell-1)}_{2,1}))\chi_\varepsilon(L(X^{(\ell-2)}_{2,3}))\chi_\varepsilon(L(X^{(\ell-2)}_{1,4}))\nonumber\\
&-\chi_\varepsilon(L(X^{(\ell-1)}_{2,3}))\,\chi_\varepsilon(L(X^{(\ell-1)}_{1,2}))\,\chi_\varepsilon(L(X^{(\ell-2)}_{2,3}))\,\chi_\varepsilon(L(X^{(\ell-2)}_{1,2})).
\end{align}}
By direct computation and induction, we obtain 
\begin{align*}
\chi_\varepsilon(L(X^{(\ell-1)}_{2,1}))\chi_\varepsilon(L(X^{(\ell-1)}_{2,3}))=\chi_\varepsilon(L(X^{(\ell-2)}_{2,3}))\left(\chi_\varepsilon(L(X^{(\ell)}_{2,1}))+\chi_\varepsilon(L(M^{2,3,\ell-2}_{1,0,1}))\right)+\chi_\varepsilon(L(X^{(\ell-1)}_{1,2})),\\
\chi_\varepsilon(L(X^{(\ell-1)}_{1,2})) \chi_\varepsilon(L(X^{(1)}_{1,0}))=\chi_\varepsilon(L(X^{(\ell)}_{1,0}))+\chi_\varepsilon(L(M^{1,4,\ell-2}_{2,1,1}))+\chi_\varepsilon(L(M^{2,2 \ell-1,1}_{1,2,\ell-2}))+\chi_\varepsilon(L(X^{(\ell-3)}_{1,4})),
\end{align*}
and for $(u,v)\in \mathcal{X}$, $w\in I$ with $u\neq w$, we have
\begin{align}\label{Eq:minimai decomposition}
\chi_\varepsilon(L(M^{w,v+3,\ell-2}_{u,v,1}))=\chi_\varepsilon(L(X^{(1)}_{u,v}))\chi_\varepsilon(L(X^{(\ell-2)}_{w,v+3}))-\chi_\varepsilon(L(X^{(\ell-3)}_{w,v+3}))-\chi_\varepsilon(L(X^{(\ell-3)}_{w,v+5})).
\end{align} 
Applying Equation (1), we obtain 
\[
\chi_\varepsilon(L(X^{(\ell-2)}_{1,2}))\chi_\varepsilon(L(X^{(\ell-2)}_{1,4}))=\chi_\varepsilon(L(X^{(\ell-1)}_{1,2}))\chi_\varepsilon(L(X^{(\ell-3)}_{1,4}))+\chi_\varepsilon(L(X^{(\ell-2)}_{2,3})).
\]
Substituting the above expressions into the right-hand of Equation (\ref{Eq:third identity}) and combining with (\ref{Eq:first equivent}), it remains to prove that
{\small
\begin{align}\label{Eq:finial identity}
&\chi_\varepsilon(L(X^{(\ell-2)}_{2,3}))\,\chi_\varepsilon(L(M^{1,4,\ell-2}_{2,1,1}))+\chi_\varepsilon(L(X^{(\ell-2)}_{2,3}))\,\chi_\varepsilon(L(M^{2,2 \ell-1,1}_{1,2,\ell-2}))-\chi_\varepsilon(L(X^{(\ell-1)}_{1,2}))\,\chi_\varepsilon(L(X^{(\ell-3)}_{2,3}))\nonumber \\
&-\chi_\varepsilon(L(X^{(\ell-1)}_{1,2}))\,\chi_\varepsilon(L(X^{(\ell-3)}_{2,5}))
+2\chi_\varepsilon(L(X^{(\ell-3)}_{1,4}))\,\chi_\varepsilon(L(X^{(\ell-2)}_{2,3})) -\chi_\varepsilon(L(X^{(\ell-1)}_{2,1}))\,\chi_\varepsilon(L(X^{(\ell-2)}_{1,4}))\nonumber\\
&-\chi_\varepsilon(L(X^{(\ell-1)}_{2,3}))\,\chi_\varepsilon(L(X^{(\ell-2)}_{1,2}))=0.
\end{align}}
Substituting Equations (1) and (\ref{Eq:minimai decomposition}) into (\ref{Eq:finial identity}) again, we obtain that (\ref{Eq:finial identity}) holds. Thus, Equation (3) is valid.
\end{proof}

Using the same methods as in Lemma \ref{Lem:exchange relations}, we obtain the following.
\begin{lemma}\label{lem: identities hold of type A2}
For $U_{\varepsilon}^{\res}(L\mathfrak{sl}_3)$, let $\varepsilon^{2\ell} = 1$ with $\ell \in \mathbb{Z}_{\ge 3}$. Then the following identities hold:  
{\small
\begin{align}
\chi_\varepsilon(L(X^{(t)}_{2,3}))\,\chi_\varepsilon(L(M^{2,2\ell+3,t-1}_{1,4,\ell-1}))
&=\chi_\varepsilon(L(X^{(t-1)}_{2,3}))\,\chi_\varepsilon(L(M^{2,2\ell+3,t}_{1,4,\ell-1}))\nonumber \\ 
&+\chi_\varepsilon(L(X^{(\ell-2)}_{1,4}))\,\chi_\varepsilon(L(X^{(t-1)}_{1,4})), \quad t\in[1,\ell-2];
\end{align}}

{\small
\begin{align}
\chi_\varepsilon(L(X^{(\ell-1)}_{2,3}))
\,\chi_\varepsilon(L(M^{2,2\ell+3,\ell-2}_{1,4,\ell-1}))=\chi_\varepsilon(L(M^{2,2\ell+3,\ell-1}_{1,4,\ell-1}))
\,\chi_\varepsilon(L(X^{(\ell-2)}_{2,3}))+\chi_\varepsilon(L((X^{(\ell-2)}_{1,4})^2));
\end{align}}

{\small
\begin{align}
\chi_\varepsilon(L(X^{(\ell-2)}_{1,4}))\,\chi_\varepsilon(L(M^{2,3,\ell-1}_{1,0,1}X^{(\ell-2)}_{2,3}))
&=\chi_\varepsilon(L(M^{2,2\ell+3,\ell-1}_{1,4,\ell-1}))\,\chi_\varepsilon(L(X^{(\ell-2)}_{2,3}))\nonumber\\ 
&+\chi_\varepsilon(L((X^{(\ell-1)}_{2,3})^2))\,\chi_\varepsilon(L(X^{(\ell-3)}_{1,4}));
\end{align}}

{\small
\begin{align}
\chi_\varepsilon(L(X^{(\ell-j+1)}_{2,3}))\,\chi_\varepsilon(L(M^{2,2\ell-2j+5,1}_{1,4,\ell-j}))
&=\chi_\varepsilon(L(X^{(\ell-j+2)}_{2,3}))\,\chi_\varepsilon(L(X^{(\ell-j)}_{1,4}))\nonumber\\
&+\chi_\varepsilon(L(X^{(\ell-j)}_{2,3}))\,\chi_\varepsilon(L(X^{(\ell-j+1)}_{1,4})),\quad j\in[3,\ell-1];
\end{align}}

{\small
\begin{align}
\chi_\varepsilon(L(X^{(2\ell-j-1)}_{1,4}))\,\chi_\varepsilon(L(M^{1,4\ell-2j+2,1}_{2,3,2\ell-j-1}))
&=\chi_\varepsilon(L(X^{(2\ell-j)}_{1,4}))\,\chi_\varepsilon(L(X^{(2\ell-j-1)}_{2,3}))\nonumber\\
&+\chi_\varepsilon(L(X^{(2\ell-j-2)}_{1,4}))\,\chi_\varepsilon(L(X^{(2\ell-j)}_{2,3})),\quad
j\in[\ell+2,2\ell-2].
\end{align}}
\end{lemma}

In the following, we present an example to illustrate the mutation sequence introduced in Theorem \ref{Th:KR mutation sequence}.
\begin{example}
Let $\varepsilon^{2 \ell}=1$ with $\ell\geq 3$. We set $\ell=3$. The initial cluster $\mathbf{x}=\{x_1, x_2, x_3, x_4\}$, where $x_1=[L(Y_{1,4}Y_{1,6}Y_{2,3}Y_{2,5})]$, $x_2=[L(Y_{2,3}Y_{2,5})]$, $x_3=[L(Y_{2,3})]$, $x_4=[L(Y_{1,4})]$. The mutation sequence is 
\[
\mu_3\mu_2\mu_1\mu_4 \mu_2\mu_1\ldots \mu_3\mu_2\mu_1\mu_4\mu_2\mu_1(\mathbf{x}),
\]
where $\mu_i$ represents the mutation at the $i$-th vertex. For convenience, we set $\boldsymbol{\mu}_a:= \mu_4\mu_2\mu_1$, $\boldsymbol{\mu}_b:= \mu_3\mu_2\mu_1$. By performing these mutations, we obtain
\begin{align*}
&\textit{initial cluster $\mathbf{x}$}:  &[L(Y_{1,4}Y_{1,6}Y_{2,3}Y_{2,5})], \quad\quad  &[L(Y_{2,3}Y_{2,5})],\quad\quad  &[L(Y_{2,3})], \quad\quad &[L(Y_{1,4})];\\
&\mu_1(\mathbf{x}):  &[L(Y_{1,2}Y_{1,4}Y_{2,3}Y_{2,5})], \quad\quad &[L(Y_{2,3}Y_{2,5})],  \quad\quad  &[L(Y_{2,3})], \quad\quad &[L(Y_{1,4})];\\
&\mu_2\mu_1(\mathbf{x}):  &[L(Y_{1,2}Y_{1,4}Y_{2,3}Y_{2,5})], \quad\quad &[L(Y_{1,2}Y_{1,4})],  \quad\quad  &[L(Y_{2,3})], \quad\quad &[L(Y_{1,4})];\\
&\boldsymbol{\mu}_a(\mathbf{x}):  &[L(Y_{1,2}Y_{1,4}Y_{2,3}Y_{2,5})], \quad\quad &[L(Y_{1,2}Y_{1,4})],  \quad\quad  &[L(Y_{2,3})], \quad\quad &[L(Y_{1,2})];\\
&\mu_1\boldsymbol{\mu}_a(\mathbf{x}):  &[L(Y_{1,2}Y_{1,4}Y_{2,1}Y_{2,3})], \quad\quad &[L(Y_{1,2}Y_{1,4})],  \quad\quad  &[L(Y_{2,3})], \quad\quad &[L(Y_{1,2})];\\
&\mu_2\mu_1\boldsymbol{\mu}_a(\mathbf{x}):  &[L(Y_{1,2}Y_{1,4}Y_{2,1}Y_{2,3})], \quad\quad &[L(Y_{2,1}Y_{2,3})],  \quad\quad  &[L(Y_{2,3})], \quad\quad &[L(Y_{1,2})];\\
&\boldsymbol{\mu}_b\boldsymbol{\mu}_a(\mathbf{x}):  &[L(Y_{1,2}Y_{1,4}Y_{2,1}Y_{2,3})], \quad\quad &[L(Y_{2,1}Y_{2,3})],  \quad\quad  &[L(Y_{2,1})], \quad\quad &[L(Y_{1,2})];\\
&\mu_1\boldsymbol{\mu}_b\boldsymbol{\mu}_a(\mathbf{x}):  &[L(Y_{1,0}Y_{1,2}Y_{2,1}Y_{2,3})], \quad\quad &[L(Y_{2,1}Y_{2,3})],  \quad\quad  &[L(Y_{2,1})], \quad\quad &[L(Y_{1,2})];\\
&\mu_2\mu_1\boldsymbol{\mu}_b\boldsymbol{\mu}_a(\mathbf{x}):  &[L(Y_{1,0}Y_{1,2}Y_{2,1}Y_{2,3})], \quad\quad &[L(Y_{1,0}Y_{1,2})],  \quad\quad  &[L(Y_{2,1})], \quad\quad &[L(Y_{1,2})];\\
&\boldsymbol{\mu}_a\boldsymbol{\mu}_b\boldsymbol{\mu}_a(\mathbf{x}):  &[L(Y_{1,0}Y_{1,2}Y_{2,1}Y_{2,3})], \quad\quad &[L(Y_{1,0}Y_{1,2})],  \quad\quad  &[L(Y_{2,1})], \quad\quad &[L(Y_{1,0})];\\
&\mu_1\boldsymbol{\mu}_a\boldsymbol{\mu}_b\boldsymbol{\mu}_a(\mathbf{x}):  &[L(Y_{1,0}Y_{1,2}Y_{2,5}Y_{2,7})], \quad\quad &[L(Y_{1,0}Y_{1,2})],  \quad\quad  &[L(Y_{2,1})], \quad\quad &[L(Y_{1,0})];\\
&\mu_2\mu_1\boldsymbol{\mu}_a\boldsymbol{\mu}_b\boldsymbol{\mu}_a(\mathbf{x}):  &[L(Y_{1,0}Y_{1,2}Y_{2,5}Y_{2,7})], \quad\quad &[L(Y_{2,5}Y_{2,7})],  \quad\quad  &[L(Y_{2,1})], \quad\quad &[L(Y_{1,0})];\\
&\boldsymbol{\mu}_b\boldsymbol{\mu}_a\boldsymbol{\mu}_b\boldsymbol{\mu}_a(\mathbf{x}):  &[L(Y_{1,0}Y_{1,2}Y_{2,5}Y_{2,7})], \quad\quad &[L(Y_{2,5}Y_{2,7})],  \quad\quad  &[L(Y_{2,5})], \quad\quad &[L(Y_{1,0})];\\
&\mu_1\boldsymbol{\mu}_b\boldsymbol{\mu}_a\boldsymbol{\mu}_b\boldsymbol{\mu}_a(\mathbf{x}):  &[L(Y_{1,4}Y_{1,6}Y_{2,5}Y_{2,7})], \quad\quad &[L(Y_{2,5}Y_{2,7})],  \quad\quad  &[L(Y_{2,5})], \quad\quad &[L(Y_{1,0})];\\
&\mu_2\mu_1\boldsymbol{\mu}_b\boldsymbol{\mu}_a\boldsymbol{\mu}_b\boldsymbol{\mu}_a(\mathbf{x}):  &[L(Y_{1,4}Y_{1,6}Y_{2,5}Y_{2,7})], \quad\quad &[L(Y_{1,4}Y_{1,6})],  \quad\quad  &[L(Y_{2,5})], \quad\quad &[L(Y_{1,0})];\\
&\boldsymbol{\mu}_a\boldsymbol{\mu}_b\boldsymbol{\mu}_a\boldsymbol{\mu}_b\boldsymbol{\mu}_a(\mathbf{x}):  &[L(Y_{1,4}Y_{1,6}Y_{2,5}Y_{2,7})], \quad\quad &[L(Y_{1,4}Y_{1,6})],  \quad\quad  &[L(Y_{2,5})], \quad\quad &[L(Y_{1,4})].
\end{align*}
All real Kirillov--Reshetikhin modules of $U_\varepsilon^\res({L\mathfrak{sl}_{3}})$ are obtained. 
Continuing the mutations, we obtain the following:
\begin{align*}
&\mu_1\boldsymbol{\mu}_a\boldsymbol{\mu}_b\boldsymbol{\mu}_a\boldsymbol{\mu}_b\boldsymbol{\mu}_a(\mathbf{x}):  &[L(Y_{1,4}Y_{1,6}Y_{2,3}Y_{2,5})], \quad\quad &[L(Y_{1,4}Y_{1,6})],\quad\quad  &[L(Y_{2,5})], \quad\quad  &[L(Y_{1,4})];\\
&\mu_2\mu_1\boldsymbol{\mu}_a\boldsymbol{\mu}_b\boldsymbol{\mu}_a\boldsymbol{\mu}_b\boldsymbol{\mu}_a(\mathbf{x}):  &[L(Y_{1,4}Y_{1,6}Y_{2,3}Y_{2,5})], \quad\quad &[L(Y_{2,3}Y_{2,5})],\quad\quad  &[L(Y_{2,5})], \quad\quad  &[L(Y_{1,4})];\\
&\boldsymbol{\mu}_b\boldsymbol{\mu}_a\boldsymbol{\mu}_b\boldsymbol{\mu}_a\boldsymbol{\mu}_b\boldsymbol{\mu}_a(\mathbf{x}):  &[L(Y_{1,4}Y_{1,6}Y_{2,3}Y_{2,5})], \quad\quad &[L(Y_{2,3}Y_{2,5})],\quad\quad  &[L(Y_{2,3})], \quad\quad  &[L(Y_{1,4})].
\end{align*}
The last one is precisely the initial cluster.
\end{example}

Finally, we prove Theorem \ref{Th:KR mutation sequence}. According to the three exchange relations in Lemma \ref{Lem:exchange relations}, a mutation at the cluster variable corresponding to a minimal affinization 
$L(M^{j,k+1,\ell-1}_{i,k,\ell-1})$
yields a cluster variable corresponding to the minimal affinization
$L(M^{j,k-1,\ell-1}_{i,k,\ell-1})$
of $U_{\varepsilon}^{\res}(L\mathfrak{sl}_3)$, where $\{i,j\}=\{1,2\}$ and $k\in\mathbb{Z}$.
Similarly, mutation at a cluster variable corresponding to a Kirillov--Reshetikhin module produces a cluster variable corresponding to another Kirillov--Reshetikhin module of $U_{\varepsilon}^{\res}(L\mathfrak{sl}_3)$. Hence, the mutations prescribed by the exchange relations in Lemma \ref{Lem:exchange relations} preserve the classes of minimal affinizations and Kirillov--Reshetikhin modules.

By performing the mutations according to the prescribed mutation sequences in Theorem \ref{Th:KR mutation sequence}, one verifies that the resulting cluster returns to the initial cluster after finitely many steps. Hence the mutation sequence is periodic. This completes the proof of Theorem \ref{Th:KR mutation sequence}.

\subsection{A proof of the first part of Gleitz' conjecture}\label{subsec: proof of isomorphism theorem}

\begin{theorem}\label{Th:isomorphic}
The generalized cluster algebra $\mathcal{A}$ defined in Definition \ref{Def:generalized cluster algebras of $A_2$} is isomorphic to $\mathcal{K}_0(\Rep U_{\varepsilon}^{\res}({L\mathfrak{sl}_{3}}))$.
\end{theorem}

\begin{proof}
Let $\Rep U_{\varepsilon}^{\res}({L\mathfrak{sl}_{3}})$ be the category of finite-dimensional $U_{\varepsilon}^{\res}({L\mathfrak{sl}_{3}})$-modules. By \cite{ FM02,FR98}, $\mathcal{K}_0(\Rep U_{\varepsilon}^{\res}({L\mathfrak{sl}_{3}}))$ is the polynomial ring in the classes of the fundamental modules of $\Rep U_{\varepsilon}^{\res}({L\mathfrak{sl}_{3}})$. Furthermore, by Theorem \ref{Th:KR mutation sequence}, the cluster algebra $\mathcal{A}$ contains all these fundamental $\varepsilon$-characters, which implies that $\mathcal{A}$ contains $\mathcal{K}_0(\Rep U_{\varepsilon}^{\res}({L\mathfrak{sl}_{3}}))$.

To prove the reverse inclusion, we will use a description of the image of the $\varepsilon$-character homomorphism as an intersection of kernels of screening operators. Similarly to \cite[Section 7.1]{FR98} and \cite [Section 5.1]{HL16}, for every $i\in I$, we have a linear operators $S_i$ from the ring $\ZZ[Y^{\pm 1}_{i,a}]^{a \in \CC^\times}_{i\in I}$ to a certain free module $\mathcal{Y}_i$ over this ring, which satisfies the Leibniz rule 
\[
S_i(xy)=xS_i(y)+yS_i(x), \quad\quad\quad x,y \in \ZZ[Y^{\pm 1}_{i,a}]^{a \in \CC^\times}_{i\in I}.
\]
Frenkel and Mukhin proved that the image of the $\varepsilon$-character homomorphism $\chi_\varepsilon$
equals the intersection of the kernels of the screening operators $S_i$, with $i\in I$, and
\[
\chi_\varepsilon: \Rep U_\varepsilon^\res({L\mathfrak{sl}_k})\longrightarrow \bigcap_{i\in I} \ker S_i
\]
is a ring isomorphism, see \cite[Proposition 3.6]{FM02}.

Let us consider a cluster variable $x$ of $\mathcal{A}$. By definition, $x$ is obtained from the initial cluster by a finite sequence of mutations $\boldsymbol{\mu}$. We want to show that $x$ belongs to $\mathcal{K}_0(\Rep U_{\varepsilon}^{\res}({L\mathfrak{sl}_{3}}))$. Since all cluster variables in the initial cluster belong to $\mathcal{K}_0(\Rep U_{\varepsilon}^{\res}(L\mathfrak{sl}_{3}))$, an induction on the length of the mutation sequence $\boldsymbol{\mu}$ shows that it suffices to consider the last exchange relation in $\boldsymbol{\mu}$. Using the exchange matrix of $\mathcal{A}$, we see that this relation is of one of the following forms:
\[
xy=M_1+M_2, \qquad \text{or} \qquad xy=M_1^3+aM_1^2M_2+bM_1M_2^2+M_2^3,
\]
where $y$, $M_1$, $M_2$, $a$, and $b$ are polynomials in the $\varepsilon$-characters of modules in $\Rep U_{\varepsilon}^{\res}(L\mathfrak{sl}_{3})$.

Assume that the last exchange relation of $\boldsymbol{\mu}$ is of the form $xy=M_1+M_2$, where $y, M_1, M_2$ are polynomials in the $\varepsilon$-characters of some modules of $\mathcal{K}_0(\Rep U_{\varepsilon}^{\res}({L\mathfrak{sl}_{3}}))$. Since $S_i$ is a derivation, we have 
\[
S_i(xy)=xS_i(y)+yS_i(x)=S_i(M_1)+S_i(M_2).
\]
Hence $S_i(x)=0$ because $S_i(y)=S_i(M_1)=S_i(M_2)=0$. 

Assume that the last exchange relation of $\boldsymbol{\mu}$ is of the form
$xy=M_1^3+aM_1^2M_2+bM_1M_2^2+M_2^3$,
where $y$, $M_1$, $M_2$, $a$, and $b$ are polynomials in the $\varepsilon$-characters of modules of $\mathcal{K}_0(\Rep U_{\varepsilon}^{\res}(L\mathfrak{sl}_{3}))$. Similarly, since $S_i$ is a derivation, we have 
\begin{align*}
S_i(xy)= xS_i(y)+yS_i(x)=&S_i(M_1)(3M^2_1+2aM_1M_2+bM^2_2)\\
 &+S_i(M_2)(3M^2_2+2bM_1M_2+aM^2_1)+M^2_1M_2S_i(a)+M_1M^2_2S_i(b).
\end{align*}
Hence $S_i(x)=0$ because $S_i(y)=S_i(M_1)=S_i(M_2)=S_i(a)=S_i(b)=0$. 

Thus, the cluster variable $x$ is annihilated by all the screening operators, so it is a polynomial in the $\varepsilon$-characters of modules of $\Rep U_{\varepsilon}^{\res}({L\mathfrak{sl}_{3}})$. Consequently, $x$ belongs to $\mathcal{K}_0(\Rep U_{\varepsilon}^{\res}({L\mathfrak{sl}_{3}}))$.
\end{proof}

In the following, we will show that $\mathcal{A}^{\rm up} \cong \mathcal{K}_0(\Rep U_{\varepsilon}^{\res}({L\mathfrak{sl}_{3}}))$. According to Remark 7.7 of \cite{Fra20}, it suffices to verify the assumptions of the Starfish Lemma \cite[Proposition 3.6]{FP16}. 
\begin{proposition}[Starfish Lemma {\cite[Proposition 3.6]{FP16}}]\label{Prop: Starfish Lemma}
Let $R$ be a finitely generated normal domain over $\mathbb{C}$, and let $\mathrm{QF}(R)$ denote the field of fractions of $R$.
Suppose that $(Q,\mathbf{z})$ is a seed in $\mathrm{QF}(R)$ satisfying the following conditions:
\begin{enumerate} 
    \item all elements of $\mathbf{z}$ belong to $R$;
    \item the cluster variables in $\mathbf{z}$ are pairwise coprime in $R$;
    \item for each cluster variable $z\in\mathbf{z}$, the seed mutation $\mu_z$ replaces $z$ with an element $z'$ such that $z'\in R$ and $z'$ is coprime to $z$.
\end{enumerate}
Then $\mathcal{A}^{\rm up}(Q,\mathbf{z}) \subseteq R$.
\end{proposition}
\begin{remark}
Proposition 3.6 of \cite{FP16} is about cluster algebra and the conclusion in Proposition 3.6 of \cite{FP16} is $\mathcal{A}(Q,\mathbf{z}) \subseteq R$. According to the discussion below Proposition 3.6 of \cite{FP16} and the discussion in Remark 7.7 of \cite{Fra20}, the results in Proposition 3.6 of \cite{FP16} work for generalized cluster algebras and the conditions (1), (2), and (3) actually imply a stronger conclusion $\mathcal{A}^{\rm up}(Q,\mathbf{z}) \subseteq R$. 
\end{remark}

\begin{lemma} \label{lem:upper generalized cluster algebra is contained in Grothendieck ring}
Let $\mathcal{A}$ be the generalized cluster algebra defined in Definition \ref{Def:generalized cluster algebras of $A_2$}. Then its upper generalized cluster algebra $\mathcal{A}^{\rm up}$ is contained in the Grothendieck ring $\mathcal{K}_0(\Rep U_{\varepsilon}^{\res}({L\mathfrak{sl}_{3}}))$. 
\end{lemma}

\begin{proof}
Let $R=\mathcal{K}_0(\Rep U_{\varepsilon}^{\res}(L\mathfrak{sl}_3))$. We now verify that the initial seed satisfies the three conditions of Starfish Lemma: Proposition \ref{Prop: Starfish Lemma}. 

{\bf Condition (1).} By Theorem \ref{Th:isomorphic}, we have $\mathcal{A} \cong R$. Hence all cluster variables belong to $R$. In particular, initial cluster variables belong to $R$ and condition (1) holds. 

{\bf Condition (2).} We now show that the initial cluster variables are pairwise coprime. The initial cluster $\mathbf{x}=\{x_1,\dots,x_{2\ell-2}\}$ is given by
\begin{align} \label{eq:initial cluster variables for type A2}
\begin{split}
x_1 =& [L(Y_{1,4}Y_{1,6}\cdots Y_{1, 2 \ell}Y_{2,2\ell+3}Y_{2,2\ell+5}\cdots Y_{2,4\ell -1})],\\
x_2 =& [L(Y_{2,2\ell+3}Y_{2,2\ell+5}\cdots Y_{2,4\ell -1})], \quad
x_3 = [L(Y_{2,2\ell+3}Y_{2,2\ell+5}\cdots Y_{2,4\ell -3})], \ \ldots, \ 
x_\ell = [L(Y_{2,2\ell+3})],\\
x_{\ell+1}& = [L(Y_{1,4}Y_{1,6}\cdots Y_{1, 2 \ell-2})], \qquad
x_{\ell+2} = [L(Y_{1,4}Y_{1,6}\cdots Y_{1, 2 \ell-4})], \,\quad\ldots, \,\quad
x_{2\ell-2} = [L(Y_{1,4})].
\end{split} 
\end{align}
The cluster variables $x_2, x_3, \ldots, x_{2\ell-2}$ correspond to Kirillov--Reshetikhin modules. These Kirillov--Reshetikhin modules are $\ell$-acyclic (see (\ref{mathbf{Y}})), and hence they are prime. For the cluster variable $x_1$, by computing the $\varepsilon$-character of the corresponding minimal affinization (see Lemma \ref{Lem: minimal affine is special}), we conclude that this module is also prime. Therefore, all initial cluster variables are prime and different from each other, and hence they are pairwise coprime.

{\bf Condition (3).} Let $x_k$ be a cluster variable in the initial seed and let $x_k'$ be the cluster variable obtained by mutating $x_k$. We will show that $x_k$ and $x'_k$ are coprime. Since we know that the initial cluster variable $x_k$ corresponds to a simple prime module, it suffices to show that $x_k'$ corresponds to a simple module and the dominant monomial of $x_k$ is not a factor of the dominant monomial of $x_k'$. According to the initial quiver in Figure \ref{fig:initial quiver of A}, we consider the following cases separately.

Case (1). By Lemma \ref{Lem:exchange relations}, for the vertex $v_1$, we know that the exchange relation is
\begin{align*}
x_1\,x'_1=x_2^3+\rho_{1,1}\,x_2^2\,x_{\ell+1}+\rho_{1,2}\,x_2\,x_{\ell+1}^2+x_{\ell+1}^3,
\end{align*} 
where $x_1$, $x_2$, $x_{\ell+1}$ are given in (\ref{eq:initial cluster variables for type A2}), $\rho_{1,1}=[L(Y_{1,0}Y_{1,2}\cdots Y_{1, 2 \ell-2})]$, $\rho_{1,2}=[L(Y_{2,1}Y_{2,3}\cdots Y_{2, 2 \ell-1})]$, and 
\[x'_1= [L(Y_{1,2}Y_{1,4}\cdots Y_{1, 2 \ell-2}Y_{2,2 \ell+3}Y_{2,2 \ell+5}\cdots Y_{2,4 \ell-1})]\in R.
\]
Since the module $L(Y_{1,4}Y_{1,6}\cdots Y_{1, 2 \ell}Y_{2,2\ell+3}Y_{2,2\ell+5}\cdots Y_{2,4\ell -1})$ corresponding to $x_1$ is prime, and its dominant monomial is not a factor of the dominant monomial of the module 
\[
L(Y_{1,2}Y_{1,4}\cdots Y_{1, 2 \ell-2}Y_{2,2 \ell+3}Y_{2,2 \ell+5}\cdots Y_{2,4 \ell-1}),
\] 
which corresponds to $x'_1$, it follows that $x_1$ and $x'_1$ are coprime.

Case (2). By Lemma \ref{lem: identities hold of type A2}, for the vertices $v_2$ and $v_{\ell+1}$, the exchange relations are respectively
\begin{align*}
x_2\,x'_2 = x_1\,x_3 + x_{\ell+1}^2, \quad 
x_{\ell+1}\,x'_{\ell+1} = x_1\,x_3 + x_2^2\,x_{\ell+2},
\end{align*}
where
\[
x'_2 = [L(Y_{1,4}\cdots Y_{1, 2 \ell}Y_{2,2 \ell+3}Y_{2,2 \ell+5}\cdots Y_{2,4 \ell-3})],\quad
x'_{\ell+1} = [L(Y_{1,0}Y_{2,3}^2Y_{2,5}^2\cdots Y_{2,2 \ell-3}^2 Y_{2,2 \ell-1})].
\]
Since the module $L(Y_{2,2\ell+3}Y_{2,2\ell+5}\cdots Y_{2,4\ell -1})$ corresponding to $x_2$ is prime, and its dominant monomial is not a factor of the dominant monomial of the module $[L(Y_{1,4}\cdots Y_{1, 2 \ell}Y_{2,2 \ell+3}Y_{2,2 \ell+5}\cdots Y_{2,4 \ell-3})]$, which corresponds to $x'_2$, it follows that $x_2$ and $x'_2$ are coprime. Similarly, since the module $L(Y_{1,4}Y_{1,6}\cdots Y_{1,2\ell-2})$ corresponding to $x_{\ell+1}$ is prime, and its dominant monomial $Y_{1,4}Y_{1,6}\cdots Y_{1,2\ell-2}$ is not a factor of the dominant monomial of the module
\[
L(Y_{1,0}Y_{2,3}^2Y_{2,5}^2\cdots Y_{2,2\ell-3}^2Y_{2,2\ell-1}),
\]
which corresponds to $x'_{\ell+1}$, we have that $x_{\ell+1}$ and $x'_{\ell+1}$ are coprime.

Case (3). By Lemma \ref{lem: identities hold of type A2}, for the vertex $v_j$, where $j\in[3,\ell-1]\cup[\ell+2,2\ell-3]$, the exchange relation is given by
\begin{align*}
x_j\,x'_j=x_{j-1}\,x_{\alpha(j)} + x_{j+1}\,x_{\beta(j)},
\end{align*}
where
\[
x_j=
\begin{cases}
[L(Y_{2,3}Y_{2,5}\cdots Y_{2, 2 \ell-2j+3})], & j\in[3,\ell-1],\\
[L(Y_{1,4}Y_{1,6}\cdots Y_{1, 4 \ell-2j})], & j\in[\ell+2,2\ell-2],
\end{cases}
\]
\[
(\alpha(j),\beta(j))=
\begin{cases}
(\ell+j-1,\ \ell+j-2), & j\in[3,\ell-1],\\
(j-\ell+2,\ j-\ell+1), & j\in[\ell+2,2\ell-2],
\end{cases}
\]
and
\[
x'_j=
\begin{cases}
[L(Y_{1,4}Y_{1,6}\cdots Y_{1, 2 \ell-2j+2}Y_{2, 2 \ell-2j+5})], & j\in[3,\ell-1],\\
[L(Y_{2,3}Y_{2,5}\cdots Y_{2, 4 \ell-2j-1}Y_{1, 4 \ell-2j+2})], & j\in[\ell+2,2\ell-2].
\end{cases}
\]
Since the module corresponding to $x_j$, $j\in[3,\ell-1]\cup[\ell+2,2\ell-3]$, is prime, and its dominant monomial is not a factor of the dominant monomial of the module corresponding to $x'_j$, it follows that $x_j$ and $x'_j$ are coprime.

Case (4). By Lemma \ref{Lem:exchange relations}, for the vertex $v_\ell$, the exchange relation is
\[
x_\ell\,x'_\ell=x_{\ell+1}+x_{2\ell-2},
\]
where $x'_\ell=[L(Y_{2,2\ell+5})]$. By the same argument as above, we have that $x_\ell$ and $x'_\ell$ are coprime.

Therefore, for each initial cluster variable $x_j$, the seed mutation $\mu_{x_j}$ replaces $x_j$ with an element $x'_j \in R$. Moreover, $x_j$ and $x'_j$ are coprime. 

Collecting the above, all conditions of the Starfish Lemma (Proposition \ref{Prop: Starfish Lemma}) are satisfied, and we conclude that $\mathcal A^{\rm up}\subseteq R$.
\end{proof}

By Theorem \ref{Th:isomorphic} and Lemma \ref{lem:upper generalized cluster algebra is contained in Grothendieck ring}, we have the following.

\begin{theorem}
Let $\varepsilon^{2 \ell}=1$ with $\ell\in\ZZ_{\geq 2}$. Then the generalized cluster algebra $\mathcal{A}$ defined in Definition~\ref{Def:generalized cluster algebras of $A_2$} coincides with its upper cluster algebra $\mathcal{A}^{\rm up}$, and they are isomorphic to the Grothendieck ring $\mathcal{K}_0(\Rep U_{\varepsilon}^{\res}(L\mathfrak{sl}_3))$.
\end{theorem}

\section{\texorpdfstring{An example for $U_\varepsilon^{\res}(L\mathfrak{sl}_4)$}{An example in type A3}} \label{sec:example in type A3}

Let $k=4$, $\ell=2$, and let $\mathcal{A}$ be the generalized cluster algebra defined in Conjecture \ref{Conj: rewrite Fraser's conjecture}. In this section, we explain that the upper generalized cluster algebra $\mathcal{A}^{\rm up}$ is isomorphic to the Grothendieck ring $K_0(\Rep U_\varepsilon^{\res}(L\mathfrak{sl}_4))$. Moreover, we expect that $\mathcal{A}$ is strictly contained in $K_0(\Rep U_\varepsilon^{\res}(L\mathfrak{sl}_4))$.

\subsection{\texorpdfstring{The generalized cluster algebra for $U_\varepsilon^{\res}(L\mathfrak{sl}_4)$ with $\ell=2$}{The generalized cluster algebra for A3 and l=2}}\label{sub:The generalized cluster algebra for A_3}
Let $\varepsilon^{2 \ell}=1$ with $\ell=2$. For $U_\varepsilon^{\res}(L\mathfrak{sl}_4)$, the generalized cluster algebra $\mathcal{A}$ constructed in Conjecture \ref{Conj: rewrite Fraser's conjecture} is as follows. The initial cluster of $\mathcal{A}$ is $\mathbf{x}=\{x_1,x_2,x_3\}$, where
\[
x_1=[L(Y_{1,2}Y_{2,1}Y_{3,0})],\quad x_2=[L(Y_{1,2}Y_{2,1})],\quad x_3=[L(Y_{1,2})].
\]
The exchange degrees are $d_1=4$ and $d_i=1$ for $i=2,3$. The exchange matrix is
\begin{equation*}
\left(\begin{array}{ccc}
0&1&-1\\
1&0&-2\\
-1&2&0
\end{array} \right).
\end{equation*}
The exchange matrix and exchange degrees correspond to the following quiver:
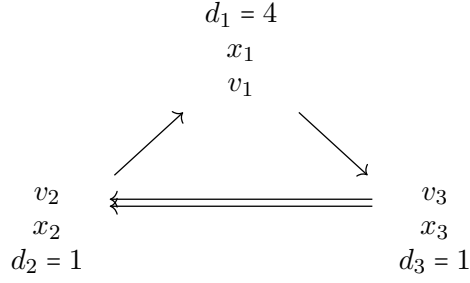
\[\begin{tikzcd}
	& \begin{array}{c} d_1=4\\x_1\\v_1 \end{array} & \\
	\begin{array}{c} v_2\\x_2\\d_2=1 \end{array} && \begin{array}{c} v_3\\x_3\\d_3=1 \end{array}
	\arrow[shift right, from=1-2, to=2-3]
	\arrow[shift right, from=2-1, to=1-2]
	\arrow[shift right=3, from=2-3, to=2-1]
	\arrow[shift right=4, from=2-3, to=2-1]
\end{tikzcd}\]
The generalized exchange relation is 
\begin{align}\label{Eq:A_3 generalized exchange relation}
\theta_1[\rho_{1}](u,v)= v^4 + \rho_{1,1}\, v^3 \,u + \rho_{1,2}\,v^2\, u^2 + \rho_{1,3}\, v\, u^3 + u^4,
\end{align}
where 
\[
u= \displaystyle \prod_{i=1}^3 x_i^{\lbrack b_{i1}\rbrack_+},\,v= \displaystyle \prod_{i=1}^3 x_i^{\lbrack -b_{i1}\rbrack_+},\,\rho_{1,1}=[L(Y_{1,0}Y_{1,2})], \, \rho_{1,2}=[L(Y_{2,1}Y_{2,3})], \text{ and }\rho_{1,3}=[L(Y_{3,0}Y_{3,2})].
\]
For $j=2$, $3$, the exchange relations are
\begin{align}\label{Eq:A_3 exchange relation 1}
\theta_j(u,v)=\displaystyle \prod_{i=1}^3 x_i^{\lbrack b_{ij}\rbrack_+} + \prod_{i=1}^3 x_i^{\lbrack -b_{ij}\rbrack_+}.
\end{align}
 
\subsection{\texorpdfstring{Mutation sequences}{Mutation sequence}}\label{sub: mutation sequence of A_3} In this subsection, we study mutation sequences of $\mathcal{A}$. Since mutation is an involution, we do not perform mutations on the same cluster variable consecutively. Moreover, we have $d_1=4$ and $d_i=1$ for $i=2,3$. Note that, under arbitrary mutations of the initial cluster, the associated quiver is unique up to isomorphism as follows:
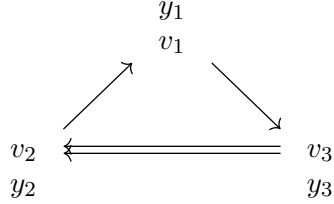
\[\begin{tikzcd}
	& \begin{array}{c} y_1\\v_1 \end{array} \\
	\begin{array}{c} v_2\\y_2 \end{array} && \begin{array}{c} v_3\\y_3 \end{array}
	\arrow[from=1-2, to=2-3]
	\arrow[from=2-1, to=1-2]
	\arrow[shift right=3, from=2-3, to=2-1]
	\arrow[shift right=2, from=2-3, to=2-1]
\end{tikzcd}\]\\
where $\{y_1,y_2,y_3\}$ is obtained from the initial cluster of $\mathcal{A}$ by a sequence of mutations.

Starting from the initial cluster variables
\[
x_1=[L(Y_{1,2}Y_{2,1}Y_{3,0})],\qquad
x_2=[L(Y_{1,2}Y_{2,1})],\qquad
x_3=[L(Y_{1,2})],
\]
we first mutate at $x_2$, and then perform an arbitrary sequence of mutations. In this situation, there are two cases as follows. Let $\mu_i$ denote mutation at the $i$-th vertex.
\begin{align*}
\text{Case (1)}.\quad & \text{initial cluster }\mathbf{x}: &[L(Y_{1,2}Y_{2,1}Y_{3,0})],\qquad &[L(Y_{1,2}Y_{2,1})],\qquad &[L(Y_{1,2})];\qquad\\
&\mu_2(\mathbf{x}):  &[L(Y_{1,2}Y_{2,1}Y_{3,0})], \qquad &[L(Y_{3,0})],\qquad&[L(Y_{1,2})];\qquad \\
&\mu_1\mu_2(\mathbf{x}):  &[L(Y_{1,2}Y_{2,3}Y_{3,0})],\qquad&[L(Y_{3,0})],\qquad&[L(Y_{1,2})];\qquad \\
&\mu_2\mu_1\mu_2(\mathbf{x}) : &[L(Y_{1,2}Y_{2,3}Y_{3,0})],\qquad &[L(Y_{1,2}Y_{2,3})],\qquad&[L(Y_{1,2})];\qquad \\
&\mu_3\mu_1\mu_2(\mathbf{x}): &[L(Y_{1,2}Y_{2,3}Y_{3,0})],\qquad &[L(Y_{3,0})],\qquad &[L(Y_{2,3}Y_{3,0})];\qquad \\
&\mu_1\mu_2\mu_1\mu_2(\mathbf{x}) : &[L(Y^3_{1,2}Y^2_{2,3}Y_{3,2})],\qquad &[L(Y_{1,2}Y_{2,3})],\qquad&[L(Y_{1,2})];\qquad \\
&\mu_3\mu_2\mu_1\mu_2(\mathbf{x}) : &[L(Y_{1,2}Y_{2,3}Y_{3,0})],\qquad &[L(Y_{1,2}Y_{2,3})],\qquad&[L(Y_{1,2}Y^2_{2,3})];\qquad \\
&\mu_1\mu_3\mu_1\mu_2(\mathbf{x}): &[L(Y_{1,0}Y^2_{2,3}Y^3_{3,0})],\qquad &[L(Y_{3,0})],\qquad &[L(Y_{2,3}Y_{3,0})];\qquad \\
&\mu_2\mu_3\mu_1\mu_2(\mathbf{x}): &[L(Y_{1,2}Y_{2,3}Y_{3,0})],\qquad &[L(Y^2_{2,3}Y_{3,0})],\qquad &[L(Y_{2,3}Y_{3,0})].\qquad 
\end{align*}
\begin{align*}
\text{Case (2)}.\quad& \text{initial cluster }\mathbf{x}:&[L(Y_{1,2}Y_{2,1}Y_{3,0})],\qquad & [L(Y_{1,2}Y_{2,1})], & [L(Y_{1,2})];\qquad\\
&\mu_2(\mathbf{x}):  & [L(Y_{1,2}Y_{2,1}Y_{3,0})], \qquad &[L(Y_{3,0})],\qquad &[L(Y_{1,2})];\qquad \\
&\mu_3\mu_2(\mathbf{x}):  & [L(Y_{1,2}Y_{2,1}Y_{3,0})],\qquad &[L(Y_{3,0})], \qquad &
[L(Y_{2,1}Y_{3,0})];\qquad \\
&\mu_1\mu_3\mu_2(\mathbf{x}):  & [L(Y_{1,0}Y^2_{2,1}Y^3_{3,0})],\qquad &[L(Y_{3,0})], \qquad & [L(Y_{2,1}Y_{3,0})];\qquad\\
&\mu_2\mu_3\mu_2(\mathbf{x}):  & [L(Y_{1,2}Y_{2,1}Y_{3,0})],\qquad &[L(Y^2_{1,2}Y_{3,0})], \qquad & [L(Y_{2,1}Y_{3,0})].\qquad
\end{align*}
Starting from these two cases, we perform arbitrary further mutations. Each mutation replaces a cluster variable by another, whose corresponding module has its dominant monomial determined by the highest exponent term in the exchange polynomial. Any further sequence of mutations produces cluster variables, which correspond to simple modules whose dominant monomials have increasing exponents. In the same way, if we first mutate cluster variable $x_1$ or $x_3$, then further arbitrary mutations give cluster variables, which correspond to simple modules whose dominant monomials also have increasing exponents.

\subsection{\texorpdfstring{$\mathcal{K}_0(\Rep U_{\varepsilon}^{\res}({L\mathfrak{sl}_{4}}))$ is isomorphic to $\mathcal{A}^{\mathrm{up}}$}{Grothendieck ring is isomorphic to upper generalized cluster algebra}}  \label{subsec:Grothendieck ring in type A3 is isomorphic to upper cluster algebra}
First, we show that all fundamental modules of $U_\varepsilon^{\res}(L\mathfrak{sl}_4)$ lie in $\mathcal{A}^{\rm up}$. In particular, we obtain $\mathcal{K}_0(\Rep U_{\varepsilon}^{\res}(L\mathfrak{sl}_4)) \subseteq \mathcal{A}^{\mathrm{up}}$. Consider the mutation at the cluster variable $x_2$ in the initial cluster $\mathbf{x}$. We have
\[
[L(Y_{1,2}Y_{2,1})][L(Y_{3,0})]
=[L(Y_{1,2}Y_{2,1}Y_{3,0})]+[L(Y_{1,2})]^2,
\]
which yields
\[
[L(Y_{3,0})]=\frac{x_1+x_3^2}{x_2}.
\]
Using $\varepsilon$-characters, we further obtain
\[
[L(Y_{2,1})][L(Y_{1,2})]=[L(Y_{1,2}Y_{2,1})]+[L(Y_{3,0})].
\]
Hence,
\[
[L(Y_{2,1})]=\frac{x_2+\frac{x_1+x_3^2}{x_2}}{x_3}=\frac{x_1+x_2^2+x_3^2}{x_2x_3}.
\]
By the definition of the upper generalized cluster algebra (see (\ref{Eq: the definition of upper generalized cluster algebra})), we conclude that
\[
[L(Y_{2,1})]\in \mathcal{A}^{\rm up}.
\]
Next, the mutation sequence $\mu_2\mu_1\mu_2(\mathbf{x})$ (refer to the fourth line in Case (1) of Section \ref{sub: mutation sequence of A_3}) yields 
\[
[L(Y_{1,2}Y_{2,3})]=\frac{u^3+\rho_{1,1}\,u^2x_3+\rho_{1,2}\,ux^2_3+\rho_{1,3}\,x^3_3+x_2x^2_3}{x_1},
\]
where $u=\frac{x_1+x^2_3}{x_2}$. Again, since 
\[
[L(Y_{2,3})][L(Y_{1,2})]=[L(Y_{1,2}Y_{2,3})]+[L(Y_{3,0})],
\]
we find that
\[
[L(Y_{2,3})]=\frac{u^3+\rho_{1,1}\,u^2x_3+\rho_{1,2}\,ux^2_3+\rho_{1,3}\,x^3_3+x_2x^2_3+x_1u}{x_1x_3}\in \mathcal{A}^{\rm up}.
\]
Furthermore, using $\varepsilon$-characters, we have
\[
[L(Y_{1,0})][L(Y_{1,2})]=[L(Y_{1,0}Y_{1,2})]+[L(Y_{2,1})]+[L(Y_{2,3})].
\]
Consequently, we can express $[L(Y_{1,0})]$ as
\begin{align*}
[L(Y_{1,0})]&=\frac{\rho_{1,1}+[L(Y_{2,1})]+[L(Y_{2,3})]}{x_3}\\
            &=\frac{\rho_{1,1}}{x_3}+\frac{x_1+x^2_2+x^2_3}{x_2x^2_3}+\frac{u^3+\rho_{1,1}\,u^2x_3+\rho_{1,2}\,ux^2_3+\rho_{1,3}\,x^3_3+x_2x^2_3+x_1u}{x_1x^2_3}\in \mathcal{A}^{\rm up}.
\end{align*}
By a similar approach, we have 
\[
[L(Y_{3,0})][L(Y_{3,2})]=[L(Y_{3,0}Y_{3,2})]+[L(Y_{2,1})]+[L(Y_{2,3})].
\]
Consequently, we can express $[L(Y_{3,2})]$ as
\begin{align*}
[L(Y_{3,2})]&=\frac{\rho_{1,3}+[L(Y_{2,1})]+[L(Y_{2,3})]}{[L(Y_{3,0})]}\\
            &=\frac{u^2+\rho_{1,1}\,ux_3+\rho_{1,2}\,x^2_3+\rho_{1,3}\,x_2x_3+x^2_2+2x_1}{x_1x_3}\in \mathcal{A}^{\rm up},
\end{align*}
where $u=\frac{x_1+x^2_3}{x_2}$.

Since $[L(Y_{1,2})]$ and $[L(Y_{3,0})]$ are cluster variables of $\mathcal{A}$, and $\mathcal{A} \subseteq \mathcal{A}^{\mathrm{up}}$, we obtain that $[L(Y_{1,2})]$,\, $[L(Y_{3,0})] \in \mathcal{A}^{\mathrm{up}}$.
Moreover, we have shown that the classes $[L(Y_{2,1})]$,\ $[L(Y_{2,3})]$,\ $[L(Y_{1,0})]$,\ and $[L(Y_{3,2})]$ also belong to $\mathcal{A}^{\mathrm{up}}$. Hence $\mathcal{A}^{\mathrm{up}}$ contains the classes of all fundamental $U_{\varepsilon}^{\res}(L\mathfrak{sl}_4)$-modules. Moreover, the Grothendieck ring
$\mathcal{K}_0(\Rep U_{\varepsilon}^{\res}(L\mathfrak{sl}_4))$ is generated as a polynomial ring by the classes of fundamental modules, it follows that
\[
\mathcal{K}_0(\Rep U_{\varepsilon}^{\res}(L\mathfrak{sl}_4))
\subseteq \mathcal{A}^{\mathrm{up}}.
\]

Let $R=\mathcal{K}_0(\Rep U_{\varepsilon}^{\res}(L\mathfrak{sl}_4))$. In the following, we prove the reverse inclusion $\mathcal{A}^{\mathrm{up}} \subseteq R$. We now verify that the initial seed satisfies the three conditions of Starfish Lemma: Proposition \ref{Prop: Starfish Lemma}.

{\bf Condition (1).} By the construction, each initial cluster variable $x_i$ is of the form $[L(m_i)]$, hence belongs to $R$. 

{\bf Condition (2).}  We now show that the initial cluster variables are pairwise coprime. The initial cluster of $\mathcal{A}$ (see Section \ref{sub:The generalized cluster algebra for A_3}) is
\[
\mathbf{x}=\{x_1,x_2,x_3\},\quad\textit{where} \quad
x_1=[L(Y_{1,2}Y_{2,1}Y_{3,0})],\quad
x_2=[L(Y_{1,2}Y_{2,1})],\quad
x_3=[L(Y_{1,2})].
\]
Since $x_3$ corresponds to a fundamental module, $x_3$ is prime. Using $\varepsilon$-characters, we  obtain that $x_1$, $x_2$ are prime. Therefore, all initial cluster variables are prime and they are different from each other, and hence they are pairwise coprime.

{\bf Condition (3).} Let $x_k$ be a cluster variable in the initial seed and let $x_k'$ be the cluster variable obtained by mutating $x_k$. We will show that $x_k$ and $x'_k$ are coprime. According to the initial quiver in Section \ref{sub:The generalized cluster algebra for A_3}, we consider the following cases separately. 

Case (1). For the vertex $v_1$, the generalized exchange relation is
\begin{align}\label{x1 of A3}
x_1\,x'_1=x_3^4+\rho_{1,1}\,x_3^3\,x_2+\rho_{1,2}\,x_3^2\,x_2^2+\rho_{1,3}\,x_3\,x_2^3+x_2^4,
\end{align}
where $\rho_{1,1}=[L(Y_{1,0}Y_{1,2})]$, $\rho_{1,2}=[L(Y_{2,1}Y_{2,3})]$, $\rho_{1,3}=[L(Y_{3,0}Y_{3,2})]$, and 
\[
x'_1=[L(Y^3_{1,2}Y^2_{2,1}Y_{3,2})].
\]
So, $x'_1$ belongs to $R$. Moreover, the module $L(Y_{1,2}Y_{2,1}Y_{3,0})$ corresponding to $x_1$ is prime, and its dominant monomial $Y_{1,2}Y_{2,1}Y_{3,0}$ does not divide the dominant monomial of the module $L(Y^3_{1,2}Y^2_{2,1}Y_{3,2})$, which corresponds to $x'_1$. Therefore, $x_1$ and $x'_1$ are coprime.

Case (2). For the vertices $v_2$ and $v_3$, the exchange relations are respectively
\begin{align*}
x_2x_2' = x_1 + x_3^2,\quad
x_3x_3' = x_1 + x_2^2,
\end{align*}
where $x'_2=[L(Y_{3,0})]$, $x'_3=[L(Y_{1,2}Y^2_{2,1})]$. So, $x'_2$, $x'_3$ belong to $R$. Since $x_2$, $x_2'$ are prime and $x_2 \ne x_2'$, it follows that $x_2$ and $x_2'$ are coprime. Since the module $L(Y_{1,2})$ corresponding to $x_3$ is prime, if $L(Y_{1,2})$ and $L(Y_{1,2}Y_{2,1}^2)$ are not coprime, then $L(Y_{1,2})$ must appear as a tensor factor of $L(Y_{1,2}Y_{2,1}^2)$. Using $\varepsilon$-characters, we obtain
\[
L(Y_{1,2}Y_{2,1}^2) \ncong L(Y_{1,2}) \otimes L(Y_{2,1}^2),\quad
L(Y_{1,2}Y_{2,1}^2) \ncong L(Y_{1,2}) \otimes L(Y_{2,1}) \otimes L(Y_{2,1}).
\]
Therefore $L(Y_{1,2})$ is not a tensor factor of $L(Y_{1,2}Y_{2,1}^2)$. It follows that $x_3=[L(Y_{1,2})]$ and $x_3'=[L(Y_{1,2}Y_{2,1}^2)]$ are coprime. 

All assumptions of Proposition \ref{Prop: Starfish Lemma} are satisfied. Therefore, $\mathcal{A}^{\mathrm{up}} \subseteq R$. Consequently, we conclude that
\[
\mathcal{K}_0(\Rep U_{\varepsilon}^{\res}(L\mathfrak{sl}_4)) \cong \mathcal{A}^{\mathrm{up}}.
\]

\subsection{\texorpdfstring{On the possible non-membership of the classes 
$[L(Y_{2,1})]$ and $[L(Y_{2,3})]$ in $\mathcal{A}$ }{On the possible non-membership of [L(Y{2,1})] and [L(Y{2,3})] in A}}
Note that the $U_{\varepsilon}^{\res}(L\mathfrak{sl}_4)$-modules $\chi_\varepsilon(L(Y_{2,1}))$ and $\chi_\varepsilon(L(Y_{2,3}))$ are not real. Indeed, we have that $\chi_\varepsilon(L(Y_{2,1})) = \chi_q(L(Y_{2,1}))|_{q=\varepsilon}$. Therefore $\chi_\varepsilon(L(Y_{2,1}))^2$ has $6^2=36$ monomials (counted with multiplicity). On the other hand, the monomials in $\chi_\varepsilon(L(Y_{2,1}^2)) = \chi_\varepsilon(L(Y_{2,1}Y_{2,5}))$ is a subset of the monomials in $\chi_q(L(Y_{2,1}Y_{2,5}))|_{q=\varepsilon}$. Apply the path description \cite{MY12} of the $q$-characters of snake modules to $\chi_q(L(Y_{2,1}Y_{2,5}))$, we find that $\chi_q(L(Y_{2,1}Y_{2,5}))|_{q=\varepsilon}$ has $35$ monomials (counted with multiplicity). It follows that $\chi_\varepsilon(L(Y_{2,1}))^2 \ne \chi_\varepsilon(L(Y_{2,1}^2))$. Similarly, $\chi_\varepsilon(L(Y_{2,3}))^2 \ne \chi_\varepsilon(L(Y_{2,3}^2))$. 
Since the modules $L(Y_{2,1})$ and $L(Y_{2,3})$ are not real, they cannot be cluster variables and they cannot be obtained by mutations from an initial seed of $\mathcal{A}$.

In Section \ref{subsec:Grothendieck ring in type A3 is isomorphic to upper cluster algebra}, we show that $[L(Y_{2,1})]$ and $[L(Y_{2,3})]$ can be expressed as Laurent polynomials in the initial cluster variables of $\mathcal{A}$. In the following, we give some indication that $[L(Y_{2,1})]$ and $[L(Y_{2,3})]$ may not admit a polynomial expression in the cluster variables of $\mathcal{A}$. Equivalently, $\chi_\varepsilon(L(Y_{2,1}))$ and $\chi_\varepsilon(L(Y_{2,3}))$ may not be generated from the $\varepsilon$-characters of the cluster variables by addition and multiplication. Consider $[L(Y_{2,1})]$ (the case of $[L(Y_{2,3})]$ is similar). We expect that the following are all possible polynomial identities satisfied by $[L(Y_{2,1})]$, where all monomials involving $[L(Y_{2,1})]$ have degree $1$:
\begin{align*} 
\chi_\varepsilon(L(Y_{2,1}))+\chi_\varepsilon(L(Y_{2,3}))
&=\chi_\varepsilon(L(Y_{1,0}))\chi_\varepsilon(L(Y_{1,2}))-\chi_\varepsilon(L(Y_{1,0}Y_{1,2})),\\
\chi_\varepsilon(L(Y_{2,1}))+\chi_\varepsilon(L(Y_{2,3}))
&=\chi_\varepsilon(L(Y_{3,0}))\chi_\varepsilon(L(Y_{3,2}))-\chi_\varepsilon(L(Y_{3,0}Y_{3,2})).
\end{align*}
In these identities, $[L(Y_{2,1}]$, $[L(Y_{2,3})]$ are not cluster variables. These identities suggest that $\chi_\varepsilon(L(Y_{2,1}))$ cannot be expressed as a polynomial in the cluster variables of $\mathcal{A}$ and we expect that $\mathcal{A} \subsetneqq \mathcal{A}^{\mathrm{up}}$.

\subsection*{Acknowledgments}
The authors are very grateful to Chris Fraser for many helpful discussions.
The work was partially supported by the National Natural Science Foundation of China (Nos. 12171213, 12271224). JR Li was supported by the Austrian Science Fund (FWF), PAT 9039323, Grant-DOI 10.55776/PAT9039323.

\end{document}